\documentclass[10pt]{article}
\usepackage{color}
\usepackage{amsmath}
\usepackage{graphicx}
\usepackage{footnote}
\usepackage{subscript}
\usepackage{amsfonts,color}
\usepackage{amsmath,amssymb}
\usepackage{epsfig}
\usepackage{amssymb} 
\usepackage{mathtools}
\usepackage{amsthm,wasysym}
\usepackage{framed}

\usepackage[english]{babel}
\usepackage[utf8]{inputenc}
\makeatletter
\usepackage{float}
\usepackage{wrapfig}
\restylefloat{figure}

\makeatletter
\let\@fnsymbol\@arabic
\makeatother

\usepackage{bbm}
\newcommand{\id}{{\boldsymbol{\mathbbm{1}}}}

\newcommand{\tr}{{\rm tr}}
\newcommand{\dev}{{\rm {dev}}}
\newcommand{\sym}{{\rm {sym}}}
\newcommand{\skw}{{\rm {skew}}}

\def\dd{\displaystyle}

\newtheorem{theorem}{Theorem}[section]
\newtheorem{lemma}[theorem]{Lemma}
\newtheorem{remark}[theorem]{Remark}

\def\barr{\begin{array}}
\usepackage{lscape}

	\def\earr{\end{array}}
\def\bec#1{\begin{equation}\label{#1}}
\def\becn{\begin{equation*}}
\def\endec{\end{equation}}
\def\endecn{\end{equation*}}
\def\dd{\displaystyle}
\def\bfm#1{\mbox{\boldmat}}
\begin{document}
	
	\title{Linear constrained {C}osserat-shell models including terms up to ${O}(h^5)$.\\ Conditional and unconditional existence and uniqueness}
\author{ 
	  Ionel-Dumitrel Ghiba\thanks{Corresponding author:  Ionel-Dumitrel Ghiba,  \ Department of Mathematics,  Alexandru Ioan Cuza University of Ia\c si,  Blvd.
	 	Carol I, no. 11, 700506 Ia\c si,
	 	Romania; and  Octav Mayer Institute of Mathematics of the
	 	Romanian Academy, Ia\c si Branch,  700505 Ia\c si, email:  dumitrel.ghiba@uaic.ro}
 	\qquad and \qquad  Patrizio Neff\,\thanks{Patrizio Neff,  \ \ Head of Lehrstuhl f\"{u}r Nichtlineare Analysis und Modellierung, Fakult\"{a}t f\"{u}r
		Mathematik, Universit\"{a}t Duisburg-Essen,  Thea-Leymann Str. 9, 45127 Essen, Germany, email: patrizio.neff@uni-due.de} 
}

\maketitle
\begin{abstract}

In this paper we  linearise the recently introduced geometrically nonlinear constrained Cosserat-shell model. In the framework of the linear constrained Cosserat-shell model, we provide a comparison of our linear models with the classical linear Koiter shell model and the ``best" first order shell model.   For all proposed linear models we show existence and uniqueness based on a Korn's inequality for surfaces.  
  \medskip
  
\end{abstract}

\tableofcontents

\section{Introduction}

The aim  of this paper is to provide the linearised formulation of the constrained geometrically nonlinear  Cosserat shell model including terms up to order $O(h^5)$ in the shell-thickness $h$ proposed in \cite{GhibaNeffPartIII}. The usage of the Cosserat directors  \cite{Cosserat08,Cosserat09} is very important in the study of thin bodies like rods, plates and shells, since they are able to capture three-dimensional effects in one or two-dimensional models, each particle being endowed with  a triad of directors.
The Cosserat approach to shell theory (also called micropolar shell theory) was initiated by the Cosserat brothers, who were the first to elaborate a rigorous study on directed media. Some theories of Cosserat surfaces have been presented in the monograph \cite{Naghdi72,Rubin00} and the linearised theory has been investigated in a large number of papers, see for instance \cite{Davini75,Birsan-IJES-07,Birsan08,Birsan-EJM/S-09}. 
The theory of simple shells (models which describe the shell-like body as a deformable surface endowed with an independent triad of orthonormal vectors connected to each point of the surface) has been given by Zhilin and Altenbach in \cite{Zhilin76,Altenbach-Zhilin-88,Altenbach04,Zhilin06} and a mathematical study of the linearised equations for this model has been presented in the papers \cite{Birsan-Alten-MMAS-10,Birsan-Alten-ZAMM-11}. We refer to the review paper \cite{Altenbach-Erem-Review} for a detailed presentation of various approaches and developments concerning Cosserat-type shell theories and to the books \cite{Ciarlet98b,Ciarlet00,Ciarlet2Diff-Geo2005} for linear shell theories in the classical elasticity framework.

The constrained geometrically nonlinear  Cosserat shell model including terms up to order $O(h^5)$ in the shell-thickness $h$ proposed in \cite{GhibaNeffPartIII} is obtained as limit case of the    novel geometrically nonlinear Cosserat introduced in \cite{Neff_plate04_cmt,GhibaNeffPartI},  i.e., letting  the Cosserat couple modulus $\mu_{\rm c}\to \infty$,  the independent Cosserat microrotation must coincide with the continuum macrorotation locally. In the constrained elastic model, the parental three-dimensional model turns into the Toupin couple stress model \cite[Eq. 11.8]{Toupin62}.  The original dimensional descent was obtained starting with a 3D-parent Cosserat model and  we  obtained a kinematical model which is equivalent to the kinematical model of 6-parameter shells. Nevertheless, the theory of 6-parameter shells was developed for shell-like bodies made of Cauchy materials, see the monographs \cite{Libai98,Pietraszkiewicz-book04} or  \cite{Eremeyev06,Pietraszkiewicz11}. 
Here, the constrained geometrically nonlinear  Cosserat shell model is considered in order to be able to compare the strain measures and the energies involved in our model with those used in other theories which do not consider independent Cosserat directors.  
 In the present paper, we give the linearised form of the constrained Cosserat shell model. We  establish that the kinematics is  consistent  with both the classical linear Koiter-Sanders-Budianski \cite{budiansky1963best} ``best'' first order shell model and the Anicic-L\'{e}ger model \cite{anicic1999formulation,anicic2001modele}.

For all the proposed models we identify the set of admissible solutions and we show existence of solution.  {Comparing the conditions imposed on the shell thickness versus the initial curvature for the $O(h^5)$-models vs. the $O(h^3)$-models, we see that the existence result for the $O(h^5)$-model needs less stringent
	assumptions.} While in the unconstrained linear Cosserat shell model there is no need for the use of  Korn-type inequalities, in the constrained linear shell models the Korn-type inequalities must be applied.

 \section{Short overview of the  geometrically nonlinear  constrained  Cosserat shell model including terms up to $O(h^5)$}\setcounter{equation}{0}\label{Intro}

\subsection{Notations}
 
   In this paper, 
 for $a,b\in\mathbb{R}^n$ we let $\bigl\langle {a},{b} \bigr\rangle _{\mathbb{R}^n}$  denote the scalar product on $\mathbb{R}^n$ with
 associated vector norm $\lVert a\rVert _{\mathbb{R}^n}^2=\bigl\langle {a},{a} \bigr\rangle _{\mathbb{R}^n}$. 
 The standard Euclidean scalar product on  the set of real $n\times  {m}$ second order tensors $\mathbb{R}^{n\times  {m}}$ is given by
 $\bigl\langle  {X},{Y} \bigr\rangle _{\mathbb{R}^{n\times  {m}}}={\rm tr}(X\, Y^T)$, and thus the  {(squared)} Frobenius tensor norm is
 $\lVert {X}\rVert ^2_{\mathbb{R}^{n\times  {m}}}=\bigl\langle  {X},{X} \bigr\rangle _{\mathbb{R}^{n\times  {m}}}$. The identity tensor on $\mathbb{R}^{n \times n}$ will be denoted by $\id_n$, so that
 ${\rm tr}({X})=\bigl\langle {X},{\id}_n \bigr\rangle $, and the zero matrix is denoted by $0_n$. We let ${\rm Sym}(n)$ and ${\rm Sym}^+(n)$ denote the symmetric and positive definite symmetric tensors, respectively.  We adopt the usual abbreviations of Lie-group theory, i.e.,
 ${\rm GL}(n)=\{X\in\mathbb{R}^{n\times n}\;|\det({X})\neq 0\}$ the general linear group ${\rm SO}(n)=\{X\in {\rm GL}(n)| X^TX=\id_n,\det({X})=1\}$ with
 corresponding Lie-algebras $\mathfrak{so}(n)=\{X\in\mathbb{R}^{n\times n}\;|X^T=-X\}$ of skew symmetric tensors
 and $\mathfrak{sl}(n)=\{X\in\mathbb{R}^{n\times n}\;| \,\tr({X})=0\}$ of traceless tensors. For all $X\in\mathbb{R}^{n\times n}$ we set ${\rm sym}\, X\,=\frac{1}{2}(X^T+X)\in{\rm Sym}(n)$, $\skw\,X\,=\frac{1}{2}(X-X^T)\in \mathfrak{so}(n)$ and the deviatoric part $\dev \,X\,=X-\frac{1}{n}\;\,\tr(X)\cdot\id_n\in \mathfrak{sl}(n)$  and we have
 the orthogonal Cartan-decomposition  of the Lie-algebra
 $
 \mathfrak{gl}(n)=\{\mathfrak{sl}(n)\cap {\rm Sym}(n)\}\oplus\mathfrak{so}(n) \oplus\mathbb{R}\!\cdot\! \id_n,$ $
 X=\dev\, \sym \,X\,+ \skw\,X\,+\frac{1}{n}\,\tr(X)\!\cdot\! \id_n\,.
 $  A matrix having the  three  column vectors $A_1,A_2, A_3$ will be written as 
 $
 (A_1\,|\, A_2\,|\,A_3).
 $ 
 We make use of the operator $\mathrm{axl}: \mathfrak{so}(3)\to\mathbb{R}^3$ associating with a skew-symmetric matrix $A\in \mathfrak{so}(3)$ the vector $\mathrm{axl}({A})\coloneqq(-A_{23},A_{13},-A_{12})^T$. The inverse operator will be denoted by ${\rm Anti}: \mathbb{R}^3\to \mathfrak{so}(3)$.
 
 For  an open domain  $\Omega\subseteq\mathbb{R}^{3}$,
 the usual Lebesgue spaces of square integrable functions, vector or tensor fields on $\Omega$ with values in $\mathbb{R}$, $\mathbb{R}^3$, $\mathbb{R}^{3\times 3}$ or ${\rm SO}(3)$, respectively will be denoted by ${\rm L}^2(\Omega;\mathbb{R})$, ${\rm L}^2(\Omega;\mathbb{R}^3)$, ${\rm L}^2(\Omega; \mathbb{R}^{3\times 3})$ and ${\rm L}^2(\Omega; {\rm SO}(3))$, respectively. Moreover, we use the standard Sobolev spaces ${\rm H}^{1}(\Omega; \mathbb{R})$ and ${\rm H}^{2,\infty}(\Omega; \mathbb{R})$ \cite{Adams75,Raviart79,Leis86}
 of functions $u$.  For vector fields $u=\left(    u_1, u_2, u_3\right)^T$ with  $u_i\in {\rm H}^{1}(\Omega)$, $i=1,2,3$,
 we define
 $
 \nabla \,u:=\left(
 \nabla\,  u_1\,|\,
 \nabla\, u_2\,|\,
 \nabla\, u_3
 \right)^T.
 $
 The corresponding Sobolev-space will be denoted by
 $
 {\rm H}^1(\Omega; \mathbb{R}^{3})$. A tensor $Q:\Omega\to {\rm SO}(3)$ having the components in ${\rm H}^1(\Omega; \mathbb{R})$ belongs to ${\rm H}^1(\Omega; {\rm SO}(3))$. In writing the norm in the corresponding  Sobolev-space we will specify the space. The space will be omitted only when the Frobenius norm or scalar product is considered.

 In the formulation of the minimization problem, for  a   $C^2(\omega)$ function $y_0:\omega\to \mathbb{R}^3$ such that ${\lVert \partial_{x_1}y_0\times \partial_{x_2}y_0\rVert}\neq 0$ we consider the  {\it Weingarten map (or shape operator)}  defined by 
 $
 {\rm L}_{y_0}\,=\, {\rm I}_{y_0}^{-1} {\rm II}_{y_0}\in \mathbb{R}^{2\times 2},
 $
 where ${\rm I}_{y_0}:=[{\nabla  y_0}]^T\,{\nabla  y_0}\in \mathbb{R}^{2\times 2}$ and  ${\rm II}_{y_0}:\,=\,-[{\nabla  y_0}]^T\,{\nabla  n_0}\in \mathbb{R}^{2\times 2}$ are  the matrix representations of the {\it first fundamental form (metric)} and the  {\it  second fundamental form} of the surface, respectively, and $n_0\,=\,\dd\frac{\partial_{x_1}y_0\times \partial_{x_2}y_0}{\lVert \partial_{x_1}y_0\times \partial_{x_2}y_0\rVert}$ is the unit normal to the curved reference surface.  
 Then, the {\it Gau{\ss} curvature} ${\rm K}$ of the surface  is determined by
 $
 {\rm K} :=\,{\rm det}{({\rm L}_{y_0})}\, 
 $
 and the {\it mean curvature} $\,{\rm H}\,$ through
 $
 2\,{\rm H}\, :={\rm tr}({{\rm L}_{y_0}}).
 $  We also use  the  tensors defined by
 \begin{align}\label{AB}
 {\rm A}_{y_0}&:=(\nabla y_0|0)\,\,[\nabla\Theta \,]^{-1}\in\mathbb{R}^{3\times 3}, \qquad \qquad 
 {\rm B}_{y_0}:=-(\nabla n_0|0)\,\,[\nabla\Theta \,]^{-1}\in\mathbb{R}^{3\times 3},
 \end{align}
 and the so-called \textit{{alternator tensor}} ${\rm C}_{y_0}$ of the surface \cite{Zhilin06}
 \begin{align}
 {\rm C}_{y_0}:=\det\nabla\Theta \,\, [	\nabla\Theta \,]^{-T}\,\begin{footnotesize}\begin{pmatrix}
 0&1&0 \\
 -1&0&0 \\
 0&0&0
 \end{pmatrix}\end{footnotesize}\,  [	\nabla\Theta \,]^{-1}.
 \end{align}
 
 We define the lifted quantities $\widehat{\rm I}_{y_0} \in \mathbb{R}^{3\times 3}$ by
 $
 \label{first_fundamental_form_lift1}
 \widehat{\rm I}_{y_0}\:\,=\,({\nabla  y_0}|n_0)^T({\nabla  y_0}|n_0)\,= {\rm I}_{y_0}^\flat+\widehat{0}_3\,$ and $\widehat{\rm II}_{y_0}\in\mathbb{R}^{3\times3}$ by
 $
 \widehat{\rm II}_{y_0}\,=\,-(\nabla y_0|n_0)^T(\nabla n_0|n_0)\,= \,{\rm II}_{y_0}^\flat-\widehat{0}_3,$
 where for a given matrix $M\in \mathbb{R}^{2\times 2}$ we employ the notations
 $
 M^\flat =\begin{footnotesize} 
 \left(\begin{array}{cc|c}
 M_{11}& M_{12}&0 \\
 M_{21}&M_{22}&0  \\
 \hline
 0&0&0
 \end{array} 
 \right)
 \end{footnotesize}
 \in \mathbb{R}^{3\times 3}$ and $
 \widehat{M} =\begin{footnotesize}
 \left(\begin{array}{cc|c}
 M_{11}& M_{12}&0 \\
 M_{21}&M_{22}&0  \\
 \hline
 0&0&1
 \end{array} 
 \right)\end{footnotesize}
 \in \mathbb{R}^{3\times 3}$. 
  Corresponding quantities are definite for any $C^2$-functions  $m:\omega\to \mathbb{R}^3$ such that ${\lVert \partial_{x_1}m\times \partial_{x_2}m\rVert}\neq 0$.

 \subsection{Shell geometry and kinematics}

 Let $\Omega_\xi\subset\mathbb{R}^3$ be a three-dimensional curved {\it shell-like thin domain}. Here, the domains $\Omega_h $ and $ \Omega_\xi $ are referred to a  fixed right Cartesian coordinate frame with unit vectors $
 \boldsymbol e_i$ along the axes $Ox_i$. A generic point of $\Omega_\xi$ will be denoted by $(\xi_1,\xi_2,\xi_3)$. The elastic material constituting the shell is assumed to be homogeneous and isotropic and the reference configuration $\Omega_\xi$ is assumed to be a natural state. 
 The deformation of the body occupying the domain $\Omega_\xi$ is described by a vector map $\varphi_\xi:\Omega_\xi\subset\mathbb{R}^3\rightarrow\mathbb{R}^3$ (\textit{called deformation}) and by a \textit{microrotation}  tensor
 $
 \overline{R}_\xi:\Omega_\xi\subset\mathbb{R}^3\rightarrow {\rm SO}(3)\, 
 $ attached at each point. 
 We denote the current configuration (deformed configuration) by $\Omega_c:=\varphi_\xi(\Omega_\xi)\subset\mathbb{R}^3$. We consider  a \textit{fictitious Cartesian (planar) configuration} of the body $\Omega_h $. This parameter domain $\Omega_h\subset\mathbb{R}^3$ is a right cylinder of the form
\begin{align}\Omega_h=\left\{ (x_1,x_2,x_3) \,\Big|\,\, (x_1,x_2)\in\omega, \,\,\,-\dfrac{h}{2}\,< x_3<\, \dfrac{h}{2}\, \right\} =\,\,\dd\omega\,\times\left(-\frac{h}{2},\,\frac{h}{2}\right),\end{align}
 where  $\omega\subset\mathbb{R}^2$ is a bounded domain with Lipschitz boundary
 $\partial \omega$ and the constant length $h>0$ is the \textit{thickness of the shell}.
 For shell--like bodies we consider   the  domain $\Omega_h $ to be {thin}, i.e. the thickness $h$ is {small}. 
 
 The diffeomorphism $\Theta:\mathbb{R}^3\rightarrow\mathbb{R}^3$ describing the reference configuration (the curved surface of the shell),  will be chosen in the specific form
 \begin{equation}\label{defTheta}
 \Theta(x_1,x_2,x_3)\,=\,y_0(x_1,x_2)+x_3\ n_0(x_1,x_2), \ \ \ \ \ \qquad  n_0\,=\,\dd\frac{\partial_{x_1}y_0\times \partial_{x_2}y_0}{\lVert \partial_{x_1}y_0\times \partial_{x_2}y_0\rVert}\, ,
 \end{equation}
 where $y_0:\omega\to \mathbb{R}^3$ is a function of class $C^2(\omega)$. If not otherwise indicated, by $\nabla\Theta$ we denote
 $\nabla\Theta(x_1,x_2,0)$.

 Now, let us  define the map
 $
 \varphi:\Omega_h\rightarrow \Omega_c,\  \varphi(x_1,x_2,x_3)=\varphi_\xi( \Theta(x_1,x_2,x_3)).
 $
 We view $\varphi$ as a function which maps the fictitious  planar reference configuration $\Omega_h$ into the deformed configuration $\Omega_c$ and it is the only one unknown of the model, contrary to the unconstrained Cosserat-shell model where the microrotation tensor field  $\overline{Q}_e:\Omega_h\rightarrow {\rm SO}(3)$, is  $\overline{Q}_e(x_1,x_2,x_3)=\overline{R}_\xi( \Theta(x_1,x_2,x_3))$ is an independent unknown, too.
 
In the constrained Cosserat-shell model, the reconstructed total deformation $\varphi_s:\Omega_h\subset \mathbb{R}^3\rightarrow \mathbb{R}^3$ of the shell-like body is assumed to be
 \begin{align}\label{ansatz}
 \varphi_s(x_1,x_2,x_3)\,=\,&m(x_1,x_2)+\bigg(x_3\varrho_m(x_1,x_2)+\dd\frac{x_3^2}{2}\varrho_b(x_1,x_2)\bigg)\overline{Q}_\infty(x_1,x_2)\nabla\Theta.e_3\, ,
 \end{align}
 where $m:\omega\subset\mathbb{R}^2\to\mathbb{R}^3$ represents the total 
 deformation of the midsurface,  $\overline{Q}_\infty:\Omega_h\rightarrow {\rm SO}(3)$ is  the constrained rotation field  and which depends on the deformation of the midsurface  by extracting the appropriate continuum rotation of the midsurface  through the polar decomposition \cite{neff2013grioli} in the sense that
 \begin{align}
\overline{Q}_{ \infty }={\rm polar}\big((\nabla  m|n) [\nabla\Theta ]^{-1}\big)=(\nabla m|n)[\nabla\Theta ]^{-1}\,\sqrt{[\nabla\Theta ]\,\widehat {\rm I}_{m}^{-1}\,[\nabla\Theta ]^{T}},
 \end{align}
  with the lifted quantity $\widehat{\rm I}_{m} \in \mathbb{R}^{3\times 3}$   given by
 $
 \widehat{\rm I}_{m}\coloneqq ({\nabla  m}|n)^T({\nabla  m}|n),\   n\,=\,\dd\frac{\partial_{x_1}m\times \partial_{x_2}m}{\lVert \partial_{x_1}m\times \partial_{x_2}m\rVert},
 $
 and \linebreak
 $\varrho_m,\,\varrho_b:\omega\subset\mathbb{R}^2\to \mathbb{R}$ are introduced to allow in principal for symmetric thickness stretch  ($\varrho_m\neq1$) and asymmetric thickness stretch ($\varrho_b\neq 0$) about the midsurface and they are given by
 \begin{align}\label{final_rho}
 \varrho_m\,=&\,1-\frac{\lambda}{\lambda+2\,\mu}[\bigl\langle  \overline{Q}_{ \infty }^T(\nabla m|0)[\nabla\Theta \,]^{-1},\id_3 \bigr\rangle -2]\;,\notag \\
 \dd\varrho_b\,=&-\frac{\lambda}{\lambda+2\,\mu}\bigl\langle  \overline{Q}_{ \infty }^T(\nabla (\,\overline{Q}_{ \infty }\nabla\Theta \,.e_3)|0)[\nabla\Theta \,]^{-1},\id_3 \bigr\rangle   \\
 &+\frac{\lambda}{\lambda+2\,\mu}\bigl\langle  \overline{Q}_{ \infty }^T(\nabla m|0)[\nabla\Theta \,]^{-1}(\nabla n_0|0)[\nabla\Theta \,]^{-1},\id_3 \bigr\rangle.\notag
 \end{align}
In this fashion, we  obtain a fully two-dimensional minimization problem in which the energy density is expressed in terms of the  following tensor fields (the same strain measures are also considered in \cite{Libai98,Pietraszkiewicz-book04,Eremeyev06,NeffBirsan13,Birsan-Neff-L54-2014} but with other motivations of their importance) on the surface $\omega\,$:
 \begin{itemize}
 	\item  the symmetric elastic shell strain tensor 	$\mathcal{E}_{ \infty }\in {\rm Sym}(3)$: \begin{align}\label{e551}
 	\mathcal{E}_{ \infty }:=& \,  \overline{Q}_{\infty}^T  (\nabla  m|\overline{Q}_{\infty}\nabla\Theta \,.e_3)[\nabla\Theta \,]^{-1}-\id_3=[{\rm polar}\big((\nabla  m|n) [\nabla\Theta ]^{-1}\big)]^T(\nabla m|n)[\nabla\Theta ]^{-1}-\id_3\\	=& \sqrt{[\nabla\Theta ]^{-T}\,\widehat{\rm I}_m\,\id_2^{\flat }\,[\nabla\Theta ]^{-1}}-
 	\sqrt{[\nabla\Theta ]^{-T}\,\widehat{\rm I}_{y_0}\,\id_2^{\flat }\,[\nabla\Theta ]^{-1}};\hspace{8cm}\notag
 	\end{align}
 	\item the non-symmetric elastic shell bending--curvature tensor $\mathcal{K}_{\infty}\not\in {\rm Sym}(3)$: \begin{align}\label{e552}
 	\mathcal{K}_{\infty}  :=&\, \Big(\mathrm{axl}(\overline{Q}_{\infty}^T\,\partial_{x_1} \overline{Q}_{\infty})\,|\, \mathrm{axl}(\overline{Q}_{\infty}^T\,\partial_{x_2} \overline{Q}_{\infty})\,|0\Big)[\nabla\Theta \,]^{-1}\notag\\=&\,\bigg(\mathrm{axl}(\, \sqrt{[\nabla\Theta ]\,\widehat{\rm I}_m^{-T}[\nabla\Theta ]^{T}}[\nabla\Theta ]^{-T}(\nabla m|n)^T\,\partial_{x_1} \Big((\nabla m|n)[\nabla\Theta ]^{-1} \sqrt{[\nabla\Theta ]\,\widehat{\rm I}_m^{-1}[\nabla\Theta ]^{T}}\Big)\big)\\
 	& \ \ \ \ \    \ \,|\, \mathrm{axl}(\, \sqrt{[\nabla\Theta ]\,\widehat{\rm I}_m^{-T}[\nabla\Theta ]^{T}}[\nabla\Theta ]^{-T}(\nabla m|n)^T\,\partial_{x_2} \Big((\nabla m|n)[\nabla\Theta ]^{-1} \sqrt{[\nabla\Theta ]\,\widehat{\rm I}_m^{-1}[\nabla\Theta ]^{T}}\Big)\big)  \,|0\bigg)[\nabla\Theta ]^{-1}.\notag
 	\end{align}
 \end{itemize}

  Other consequences, besides the continuum rotation  constraint $
  {Q}_{ \infty }={\rm polar}\big((\nabla  m|n) [\nabla\Theta ]^{-1}\big)$ of the constrained Cosserat-shell model are 
  \begin{align}\label{conds1}
  \mathcal{E}_{ \infty } \, {\rm B}_{y_0} +  {\rm C}_{y_0} \mathcal{K}_{ \infty }\stackrel{!}{\in}{\rm Sym}(3) & \Leftrightarrow\ \ \sqrt{[\nabla\Theta \,]^{-T}\,\widehat{\rm I}_m\,[\nabla\Theta \,]^{-1}}\,  [\nabla\Theta \,]\Big({\rm L}_{y_0}^\flat -{\rm L}_m^\flat\Big)[\nabla\Theta \,]^{-1}\notag\\&\qquad \stackrel{!}{=}[\nabla\Theta \,]^{-T}\Big(({\rm L}_{y_0}^\flat)^T-({\rm L}_{m}^\flat)^T\Big)[\nabla\Theta \,]^{T}\sqrt{[\nabla\Theta \,]^{-T}\,\widehat{\rm I}_m\,[\nabla\Theta \,]^{-1}},
  \end{align}
  and
  \begin{align}\label{conds2}
  (  \mathcal{E}_{ \infty } \, {\rm B}_{y_0} +  {\rm C}_{y_0} \mathcal{K}_{ \infty } )   {\rm B}_{y_0}\stackrel{!}{\in}{\rm Sym}(3) &\Leftrightarrow\ \ 
  \sqrt{[\nabla\Theta \,]^{-T}\,\widehat{\rm I}_m\,[\nabla\Theta \,]^{-1}}\,  [\nabla\Theta \,]\Big({\rm L}_{y_0}^\flat - {\rm L}_m^\flat\Big){\rm L}_{y_0}^\flat[\nabla\Theta \,]^{-1}\vspace{1.5mm}\\
  &\qquad \stackrel{!}{=} [\nabla\Theta \,]^{-T} ({\rm L}_{y_0}^\flat)^T\Big(({\rm L}_{y_0}^\flat)^T-({\rm L}_{m}^\flat)^T\Big))[\nabla\Theta \,]^{T}\sqrt{[\nabla\Theta \,]^{-T}\,\widehat{\rm I}_m\,[\nabla\Theta \,]^{-1}}\,.\notag
  \end{align}
These restrictions remain additional constraints in the minimization problem and they assure\footnote{	The terms of order $  O(x_3^3) $ are not relevant here, since we have taken a quadratic ansatz for the deformation.} that 	the (through the thickness reconstructed) strain tensor  \begin{align}\label{extE}
\widetilde{\mathcal{E}}_{s}  \; =\; &\,\quad \,\,
1\,\Big[  \mathcal{E}_{ \infty } - \frac{\lambda}{\lambda+2\mu}\,\tr( \mathcal{E}_{ \infty } )\; (0|0|n_0)\, (0|0|n_0)^T  \Big]
\notag\vspace{2.5mm}\\
& 
+x_3\Big[ (\mathcal{E}_{ \infty } \, {\rm B}_{y_0} +  {\rm C}_{y_0} \mathcal{K}_{ \infty }) -
\frac{\lambda}{(\lambda+2\mu)}\, {\rm tr}  (\mathcal{E}_{ \infty } {\rm B}_{y_0} + {\rm C}_{y_0}\mathcal{K}_{ \infty } )\; (0|0|n_0)\,  (0|0|n_0)^T  \Big]
\vspace{2.5mm}\\
& 
+x_3^2\Big[\,(\mathcal{E}_{ \infty } \, {\rm B}_{y_0} +  {\rm C}_{y_0} \mathcal{K}_{ \infty }) {\rm B}_{y_0} \Big]\;+\; O(x_3^3)\notag
\end{align}
remains symmetric.
\subsection{Variational problem for the constrained nonlinear Cosserat-shell model}

 The  variational problem for the constrained Cosserat $O(h^5)$-shell model (see  \cite{GhibaNeffPartIII})  is  to find a deformation of the midsurface
 $m:\omega\subset\mathbb{R}^2\to\mathbb{R}^3$  minimizing on $\omega$ the  functional
 \begin{align}\label{minvarmc}
 I= \int_{\omega}   \,\, \Big[ & \,\Big(h+{\rm K}\,\dfrac{h^3}{12}\Big)\,
 W_{{\rm shell}}^{\infty}\big(    \sqrt{[\nabla\Theta ]^{-T}\,\widehat{\rm I}_m\,\id_2^{\flat }\,[\nabla\Theta ]^{-1}}-
 \sqrt{[\nabla\Theta ]^{-T}\,\widehat{\rm I}_{y_0}\,\id_2^{\flat }\,[\nabla\Theta ]^{-1}} \big)\vspace{2.5mm}\notag\\    
 & +   \Big(\dfrac{h^3}{12}\,-{\rm K}\,\dfrac{h^5}{80}\Big)\,
 W_{{\rm shell}}^{\infty}  \big(   \sqrt{[\nabla\Theta ]^{-T}\,\widehat{\rm I}_m\,[\nabla\Theta ]^{-1}}\,  [\nabla\Theta ]\Big({\rm L}_{y_0}^\flat - {\rm L}_m^\flat\Big)[\nabla\Theta ]^{-1}\big) \notag \\\hspace*{-2.5cm}&
 -\dfrac{h^3}{3} \mathrm{ H}\,\mathcal{W}_{{\rm shell}}^{\infty}  \big(  \sqrt{[\nabla\Theta ]^{-T}\,\widehat{\rm I}_m\,\id_2^{\flat }\,[\nabla\Theta ]^{-1}}-
 \sqrt{[\nabla\Theta ]^{-T}\,\widehat{\rm I}_{y_0}\,\id_2^{\flat }\,[\nabla\Theta ]^{-1}} ,\notag\\
 &\qquad \qquad \qquad\quad\   \sqrt{[\nabla\Theta ]^{-T}\,\widehat{\rm I}_m\,[\nabla\Theta ]^{-1}}\,  [\nabla\Theta ]\Big({\rm L}_{y_0}^\flat - {\rm L}_m^\flat\Big)[\nabla\Theta ]^{-1} \big)\\&+
 \dfrac{h^3}{6}\, \mathcal{W}_{{\rm shell}}^{\infty}  \big(  \sqrt{[\nabla\Theta ]^{-T}\,\widehat{\rm I}_m\,\id_2^{\flat }\,[\nabla\Theta ]^{-1}}-
 \sqrt{[\nabla\Theta ]^{-T}\,\widehat{\rm I}_{y_0}\,\id_2^{\flat }\,[\nabla\Theta ]^{-1}} ,\notag
 \\
 &\qquad \qquad \qquad\ \ \sqrt{[\nabla\Theta ]^{-T}\,\widehat{\rm I}_m\,[\nabla\Theta ]^{-1}}\,  [\nabla\Theta ]\Big({\rm L}_{y_0}^\flat - {\rm L}_m^\flat\Big){\rm L}_{y_0}^\flat[\nabla\Theta ]^{-1}\big)\vspace{2.5mm}\notag\notag\\&+ \,\dfrac{h^5}{80}\,\,
 W_{\mathrm{mp}}^{\infty} \big(\sqrt{[\nabla\Theta ]^{-T}\,\widehat{\rm I}_m\,[\nabla\Theta ]^{-1}}\,  [\nabla\Theta ]\Big({\rm L}_{y_0}^\flat - {\rm L}_m^\flat\Big){\rm L}_{y_0}^\flat[\nabla\Theta ]^{-1}\,\big)\notag\vspace{2.5mm}\notag
 \\
 &+ \Big(h-{\rm K}\,\dfrac{h^3}{12}\Big)\,
 W_{\mathrm{curv}}\big(  \mathcal{K}_{ \infty } \big)    +  \Big(\dfrac{h^3}{12}\,-{\rm K}\,\dfrac{h^5}{80}\Big)\,
 W_{\mathrm{curv}}\big(  \mathcal{K}_{ \infty }   {\rm B}_{y_0} \,  \big)  + \,\dfrac{h^5}{80}\,\,
 W_{\mathrm{curv}}\big(  \mathcal{K}_{ \infty }   {\rm B}_{y_0}^2  \big)
 \Big] \,{\rm det}\nabla \Theta        \,\mathrm{\rm da}\notag\\&\qquad - \overline{\Pi}(m,{Q}_{ \infty }),\notag
 \end{align}
 such that
 \begin{align}\label{contsym}
 \mathcal{E}_{ \infty } \, {\rm B}_{y_0}& +  {\rm C}_{y_0} \mathcal{K}_{ \infty }\stackrel{!}{\in}{\rm Sym}(3)\ \vspace{1.5mm}\notag\\ & \Leftrightarrow\ \ \sqrt{[\nabla\Theta \,]^{-T}\,\widehat{\rm I}_m\,[\nabla\Theta \,]^{-1}}\,  [\nabla\Theta \,]\Big({\rm L}_{y_0}^\flat -{\rm L}_m^\flat\Big)[\nabla\Theta \,]^{-1}\notag\\&\qquad \stackrel{!}{=}[\nabla\Theta \,]^{-T}\Big(({\rm L}_{y_0}^\flat)^T-({\rm L}_{m}^\flat)^T\Big)[\nabla\Theta \,]^{T}\sqrt{[\nabla\Theta \,]^{-T}\,\widehat{\rm I}_m\,[\nabla\Theta \,]^{-1}},\notag\vspace{1.5mm}\notag\\
 (  \mathcal{E}_{ \infty } \,& {\rm B}_{y_0} +  {\rm C}_{y_0} \mathcal{K}_{ \infty } )   {\rm B}_{y_0}\stackrel{!}{\in}{\rm Sym}(3)\vspace{1.5mm}\\ &\Leftrightarrow\ \ 
 \sqrt{[\nabla\Theta \,]^{-T}\,\widehat{\rm I}_m\,[\nabla\Theta \,]^{-1}}\,  [\nabla\Theta \,]\Big({\rm L}_{y_0}^\flat - {\rm L}_m^\flat\Big){\rm L}_{y_0}^\flat[\nabla\Theta \,]^{-1}\notag\vspace{1.5mm}\notag\\
 &\qquad \stackrel{!}{=}  [\nabla\Theta \,]^{-T} ({\rm L}_{y_0}^\flat)^T\Big(({\rm L}_{y_0}^\flat)^T-({\rm L}_{m}^\flat)^T\Big))[\nabla\Theta \,]^{T}\sqrt{[\nabla\Theta \,]^{-T}\,\widehat{\rm I}_m\,[\nabla\Theta \,]^{-1}}\,,\notag
 \end{align}
 where 
 \begin{align} 
 \mathcal{K}_{ \infty } & = \, \Big(\mathrm{axl}({Q}_{ \infty }^T\,\partial_{x_1} {Q}_{ \infty })\,|\, \mathrm{axl}({Q}_{ \infty }^T\,\partial_{x_2} {Q}_{ \infty })\,|0\Big)[\nabla\Theta ]^{-1}, \vspace{2.5mm}\notag\\
 {Q}_{ \infty }&={\rm polar}\big((\nabla  m|n) [\nabla\Theta ]^{-1}\big)=(\nabla m|n)[\nabla\Theta ]^{-1}\,\sqrt{[\nabla\Theta ]\,\widehat {\rm I}_{m}^{-1}\,[\nabla\Theta ]^{T}},\notag\\
 W_{{\rm shell}}^{\infty}(  S)  &=   \mu\,\lVert\,   S\rVert^2  +\,\dfrac{\lambda\,\mu}{\lambda+2\mu}\,\big[ \mathrm{tr}   \, (S)\big]^2,\qquad 
 \mathcal{W}_{{\rm shell}}^{\infty}(  S,  T) =   \mu\,\bigl\langle  S,   T\bigr\rangle+\,\dfrac{\lambda\,\mu}{\lambda+2\mu}\,\mathrm{tr}  (S)\,\mathrm{tr}  (T), 
 \\
 W_{\mathrm{mp}}^{\infty}(  S)&= \mu\,\lVert  S\rVert^2+\,\dfrac{\lambda}{2}\,\big[ \mathrm{tr}\,   (S)\big]^2 \qquad \quad\forall \ S,T\in{\rm Sym}(3), \notag\vspace{2.5mm}\\
 W_{\mathrm{curv}}(  X )&=\mu\,{\rm L}_c^2\left( b_1\,\lVert \dev\,\sym \, X\rVert^2+b_2\,\lVert\skw \,X\rVert^2+b_3\,
 [\tr(X)]^2\right) \quad\qquad \forall\  X\in\mathbb{R}^{3\times 3}.\notag
 \end{align}

 The parameters $\mu$ and $\lambda$ are the usual \textit{Lam\'e constants}
 of classical isotropic elasticity, $\kappa=\frac{2\,\mu+3\,\lambda}{3}$ is the \textit{infinitesimal (elastic) bulk modulus}, $b_1, b_2, b_3$ are \textit{non-dimensional constitutive curvature coefficients (weights)} and ${L}_{\rm c}>0$ introduces an \textit{{internal length} } which is {characteristic} for the material, e.g., related to the grain size in a polycrystal. The
 internal length ${L}_{\rm c}>0$ is responsible for \textit{size effects} in the sense that smaller samples are relatively stiffer than
 larger samples. If not stated otherwise, we assume that $\mu>0$, $\kappa>0$, $b_1>0$, $b_2>0$, $b_3> 0$. All the constitutive coefficients  are coming from the three-dimensional Cosserat formulation (in the constrained model from the three-dimensional Toupin couple stress model, too), without using any a posteriori fitting of some two-dimensional constitutive coefficients.

 The potential of applied external loads $ \overline{\Pi}(m,{Q}_{ \infty }) $ appearing in the variational problem is expressed by 
 \begin{align}\label{e4o}
 \overline{\Pi}(m,{Q}_{ \infty })\,=\,& \, \Pi_\omega(m,{Q}_{ \infty }) + \Pi_{\gamma_t}(m,{Q}_{ \infty })\,,\qquad\textrm{with}   \\
 \Pi_\omega(m,{Q}_{ \infty }) \,=\,& \dd\int_{\omega}\bigl\langle  {f}, u \bigr\rangle \, da + \Lambda_\omega({Q}_{ \infty })\qquad \text{and}\qquad
 \Pi_{\gamma_t}(m,{Q}_{ \infty })\,=\, \dd\int_{\gamma_t}\bigl\langle  {t},  u \bigr\rangle \, ds + \Lambda_{\gamma_t}({Q}_{ \infty })\,,\notag
 \end{align}
 where $ u(x_1,x_2) \,=\, m(x_1,x_2)-y_0(x_1,x_2) $ is the displacement vector of the midsurface,  $\Pi_\omega(m,{Q}_{ \infty })$ is the potential of the external surface loads $f$, while  $\Pi_{\gamma_t}(m,{Q}_{ \infty })$ is the potential of the external boundary loads $t$.  The functions $\Lambda_\omega\,, \Lambda_{\gamma_t} : {\rm L}^2 (\omega, \textrm{SO}(3))\rightarrow\mathbb{R} $ are expressed in terms of the loads from the three-dimensional parental variational problem, see \cite{GhibaNeffPartI}, and they are assumed to be continuous and bounded operators. Here, $ \gamma_t $ and $ \gamma_d $ are nonempty subsets of the boundary of $ \omega $ such that $   \gamma_t \cup \gamma_d= \partial\omega $ and $ \gamma_t \cap \gamma_d= \emptyset $\,. On $ \gamma_t $ we have considered traction boundary conditions, while on $ \gamma_d $ we have the Dirichlet-type boundary conditions: \begin{align}\label{boundary}
 m\big|_{\gamma_d}&=m^*, \qquad \qquad 
{Q}_{ \infty }Q_0.e_3\big|_{ \gamma_d}=\,\dd\frac{\partial_{x_1}m^*\times \partial_{x_2}m^*}{\lVert \partial_{x_1}m^*\times \partial_{x_2}m^*\rVert },
 \end{align}
 where the boundary conditions are to be understood in the sense of traces.

 \begingroup
 \allowdisplaybreaks
 Since we would like to obtain a compatible form of our linear model with the classical linear Koiter model presented in Section \ref{Km}, before performing the linearisation,  let us remark that we can express the nonlinear strain tensors in the form
 \begin{align}\label{eq5}
 \mathcal{E}_{\infty}=&\quad\ \, [\nabla\Theta \,]^{-T}
 \left[
 (\overline{Q}_{\infty} \nabla y_0)^{T} \nabla m- {\rm I}_{y_0}\right]^\flat [\nabla\Theta \,]^{-1} \notag=\ \,
 [\nabla\Theta \,]^{-T}
 \mathcal{G}_{\infty} ^\flat [\nabla\Theta \,]^{-1},
 \vspace{6pt}\notag\\
 \mathrm{C}_{y_0} \mathcal{K}_{\infty} = &\quad\ \, [\nabla\Theta \,]^{-T}
 \left[
 (\overline{Q}_{\infty} \nabla y_0)^{T} \nabla (\overline{Q}_{\infty} n_0)+ {\rm II}_{y_0} \right]^\flat [\nabla\Theta \,]^{-1}=\, -[\nabla\Theta \,]^{-T}
 \mathcal{R}_{\infty} ^\flat [\nabla\Theta \,]^{-1},\notag\\
 \mathrm{C}_{y_0} \mathcal{K}_{\infty}{\rm B}_{y_0} = & \,- [\nabla\Theta \,]^{-T}
 \left[
 \mathcal{R}_{\infty}\, {\rm L}_{y_0}\right]^\flat[\nabla\Theta \,]^{-1},
 \vspace{10pt}\\\notag
 \mathrm{C}_{y_0} \mathcal{K}_{e,s} {\rm B}^2_{y_0} = &\, - [\nabla\Theta \,]^{-T}
 \left[
 \mathcal{R}_{\infty}\, {\rm L}^2_{y_0}\right
 ]^\flat [\nabla\Theta \,]^{-1},
 \vspace{10pt}\\
 \mathcal{E}_{\infty} {\rm B}_{y_0}  + \mathrm{C}_{y_0} \mathcal{K}_{e,s} 
 = &\, -[\nabla\Theta \,]^{-T}
 \left[
 \mathcal{R}_{\infty}-\mathcal{G}_{\infty}\,{\rm L}_{y_0}\right]^\flat\ [\nabla\Theta \,]^{-1}\notag,\\
 \mathcal{E}_{\infty}{\rm B}_{y_0}^2  + \mathrm{C}_{y_0} \mathcal{K}_{\infty} {\rm B}_{y_0}
 = &\, -[\nabla\Theta \,]^{-T}
 \left[
 (\mathcal{R}_{\infty} -\mathcal{G}_{\infty} \,{\rm L}_{y_0})\,{\rm L}_{y_0}\right]^\flat [\nabla\Theta \,]^{-1}\notag
 ,
 \end{align}
 \endgroup
 where
 \begin{align}\label{eq41}
 \mathcal{G}_{\infty} :=&\, (\overline{Q}_{\infty} \nabla y_0)^{T} \nabla m- {\rm I}_{y_0}\in {\rm Sym}(2)\qquad\qquad\qquad\qquad\quad\ \ \ \  \textrm{\it the change of metric tensor},
 \\
 \mathcal{R}_{\infty} :=& \, -(\overline{Q}_{\infty} \nabla y_0)^{T} \nabla (\overline{Q}_{\infty} n_0)- {\rm II}_{y_0}\not\in {\rm Sym}(2)
 \quad  \qquad\qquad \ \ \, \,\,\textrm{\it the bending  strain tensor}.
 \notag\end{align}
 The nonsymmetric quantity $ \mathcal{R}_{\infty}-\mathcal{G}_{\infty} \,{\rm L}_{y_0}$ represents {\it the change of curvature} tensor. The choice of this name will be justified  in a forthcoming paper \cite{GhibaNeffPartVI}. For now, we just mention that 
 the definition of $\mathcal{G}_{\infty}$ is related to the classical {\it change of metric }  tensor in the Koiter model
 \begin{align} \mathcal{G}_{\mathrm{Koiter}} := \dfrac12 \big[ ( \nabla m)^{T} \nabla m- {\rm I}_{y_0}\big]=\dfrac12\,({\rm I}_m-{\rm I}_{y_0})\in {\rm Sym}(2),\end{align}
 while the bending strain tensor $\mathcal{R}_{\infty}$ may be compared  with the classical {\it bending strain tensor}  in the Koiter model
 \begin{align} \mathcal{R}_{\rm{Koiter}} :=  -(\nabla m)^{T} \nabla n- {\rm II}_{y_0}={\rm II}_{m}-{\rm II}_{y_0}\in {\rm Sym}(2).\end{align}
 We also note  that the constrained Cosserat-shell model is not able to reflect the effect of the transverse shear vector $\mathcal{T}_\infty:= \, (\overline{Q}_{\infty}  n_0)^T (\nabla m) $, since it vanishes as is intended.

 In the constrained nonlinear Cosserat-shell model the tensor $ \mathcal{G}_\infty $ is symmetric (since $ \mathcal{E}_\infty$ is symmetric). The same will be true in the linear model, too.

    In our model the total energy is not simply the sum of  energies coupling the membrane and the change of curvature effects, respectively.  Two further coupling energies are still present, see the lines 3-6 from \eqref{minvarmc} and \eqref{eq5}, and they result directly from the dimensional reduction of the variational problem from  geometrically nonlinear three-dimensional Cosserat elasticity.  Beside this, the bending and the drilling effects are present in the model, see the 8th line from \eqref{minvarmc} and \eqref{eq5}, due to the parental three-dimensional  curvature energy, and  not derived to the parental three-dimensional quadratic energy depending on the Biot-strain tensor. 
    
 	Our model   is constructed in \cite{GhibaNeffPartI} under the following assumptions upon the thickness 
 	\begin{align}\label{ch5in}
 	h\,\max \{\sup_{(x_1,x_2)\in {\omega}}|{\kappa_1}|,\sup_{(x_1,x_2)\in {\omega}}|{\kappa_2}|\}<2\end{align}
 	where  ${\kappa_1}$ and ${\kappa_2}$ denote  the {\rm principal curvatures} of the surface.  {The hypothesis \eqref{ch5in} which guarantees the local injectivity of the parametrization of the shell
 	represents the key point of the  works \cite{anicic2001modele,anicic1999formulation} on strongly curved shells.}

The geometrically nonlinear model admits global minimizers for   materials with the Poisson ratio  $\nu=\frac{\lambda}{2(\lambda+\mu)}$ and Young's modulus $ E=\frac{\mu(3\,\lambda+2\,\mu)}{\lambda+\mu}$  such that
$
-\frac{1}{2}<\nu<\frac{1}{2}$ \text{and} $E>0\, 
$  \cite{GhibaNeffPartII}.  
Under these assumptions on the constitutive coefficients, together with the positivity of  $b_1$, $b_2$ and $b_3$, and the orthogonal Cartan-decomposition  of the Lie-algebra
$
\mathfrak{gl}(3)$ 
it follows that there exists positive constants  $c_1^+, c_2^+, C_1^+$ and $C_2^+$  such that for all $X\in \mathbb{R}^{3\times 3}$ the following inequalities hold
\begin{align}\label{pozitivdef}
C_1^+ \lVert S\rVert ^2&\geq\, {W}_{\mathrm{shell}}^{\infty}( S) \geq\, c_1^+ \lVert  S\rVert ^2 \qquad \qquad \ \forall \, S\,\in{\rm Sym}(3),\notag\\
C_2^+ \lVert X \rVert ^2 &\geq\, W_{\mathrm{curv}}(  X )
\geq\,  c_2^+ \lVert X \rVert ^2\qquad \qquad  \forall \, X\in\mathbb{R}^{3\times 3}.
\end{align}

Step by step, in the next sections we will see that our models generalize a whole family of models. In the end, the linearization of the deformation measures of our constrained Cosserat-shell model are  those  preferred  in the later works by Sanders and Budiansky \cite{budiansky1962best,budiansky1963best} and by Koiter and Simmonds \cite{koiter1973foundations}, who called the obtained theory  the ``best first-order linear elastic shell theory'' but our change of curvature measure will be that proposed by Anicic and  {L\'{e}ger} \cite{anicic1999formulation,anicic2001modele}.

\section{The classical linear (first) Koiter model and the corresponding existence results}\setcounter{equation}{0}\label{Km}
According to \cite[page 344]{Ciarlet00}, \cite[page 154]{ciarlet2005introduction}  in the linear (first) Koiter model, the variational problem is to find a midsurface displacement vector field
$v:\omega\subset\mathbb{R}^2\to\mathbb{R}^3$  minimizing\footnote{Observe that ${\rm det}(\nabla y_0|n_0)=\sqrt{\det ([\nabla y_0]^T\nabla y_0)}$ is independent of $n_0$.}
\begin{align}\label{Ap7matrix1}
\dd\int_\omega &\bigg\{h\bigg(
\mu\rVert    [\nabla\Theta]^{-T} [\mathcal{G}_{\rm{Koiter}}^{\rm{lin}}]^\flat [\nabla\Theta]^{-1}\rVert^2  +\dfrac{\,\lambda\,\mu}{\lambda+2\,\mu} \, \mathrm{tr} \Big[ [\nabla\Theta]^{-T} [\mathcal{G}_{\rm{Koiter}}^{\rm{lin}}]^\flat  [\nabla\Theta]^{-1}\Big]^2\bigg) \vspace{6pt} \\
&+\dd\frac{h^3}{12}\bigg(
\mu\rVert    [\nabla\Theta]^{-T} \big(\mathcal{R}_{\rm{Koiter}}^{\rm{lin}}\big)^\flat  [\nabla\Theta]^{-1}\rVert^2 +\dfrac{\,\lambda\,\mu}{\lambda+2\,\mu} \, \mathrm{tr} \Big[ [\nabla\Theta]^{-T} \big(\mathcal{R}_{\rm{Koiter}}^{\rm{lin}}\big)^\flat   [\nabla\Theta]^{-1}\Big]^2\bigg)\bigg\}\,{\rm det}(\nabla y_0|n_0)\, {\rm d}a,\notag
\end{align}
where the strain measures \cite{Ciarlet00}  are given by
\begin{equation}
\label{equ11}
\mathcal{G}_{\rm{Koiter}}^{\rm{lin}} \,\,:=\frac{1}{2}\big[{\rm I}_m - {\rm I}_{y_0}\big]^{\rm{lin}}= \,\,\frac12\;\big[  (\nabla y_0)^{T}(\nabla v) +  (\nabla v)^T(\nabla y_0)\big]
= \sym\big[ (\nabla y_0)^{T}(\nabla v)\big]\in {\rm Sym}(2)
\;
\end{equation}
and  {\cite{blouza1994lemme}}
\begin{align}
\label{equ12}
\mathcal{R}_{\rm{Koiter}}^{\rm{lin}} \,\,:&= \,\, \big[{\rm II}_m - {\rm II}_{y_0}\big]^{\rm{lin}} \,= 	 \Big( \bigl\langle n_0 ,  \partial_{x_\alpha  x_\beta}\,v- \dd\sum_{\gamma=1,2}\Gamma^\gamma_{\alpha \beta}\,\partial_{x_\gamma}\,v\bigr\rangle a^\alpha\,\Big)_{\alpha\beta}{\in {\rm Sym}(2)}.
\end{align}
Here, and in the rest of the paper, $a_1, a_2, a_3\,$ denote the columns of $\nabla\Theta \,$, while  $a^1, a^2, a^3\,$ denote the rows of $[\nabla\Theta \,]^{-1}$, i.e., 
\begin{align}
\nabla\Theta \,&\,=\,(\nabla y_0|\,n_0)\,=\,(a_1|\,a_2|\,a_3),\qquad \qquad \quad  [\nabla\Theta \,]^{-1}\,=\,(a^1|\,a^2|\,a^3)^T.
\end{align}
In fact,   $a_1, a_2\,$ are the covariant base vectors and $ a^1, a^2\,$ are the contravariant base vectors in the tangent plane given by
\begin{align}
a_\alpha:=\,\partial_{x_\alpha}y_{0},\qquad \qquad \qquad  \langle a^\beta, a_\alpha\rangle \,=\,\delta_\alpha^\beta,\qquad \qquad  \alpha,\beta=1,2, 
\quad \text{and}\qquad \qquad a_3\,= a^3=n_0.
\end{align}
 The following relations hold \cite[page 95]{Ciarlet00}: 
\begin{align}
\lVert a_1\times a_2\rVert=\sqrt{\det {\rm I}_{y_0}},\quad  \qquad  a_3\times a_1=\sqrt{\det {\rm I}_{y_0}}\, a^2, \quad \qquad a_2\times a_3=\sqrt{\det {\rm I}_{y_0}}\, a^1.
\end{align}
The expression of $\mathcal{R}_{\rm{Koiter}}^{\rm{lin}}$ involves the Christoffel  symbols
$
\Gamma^\gamma_{\alpha\beta}
$  on the surface given by
$
\Gamma^\gamma_{\alpha\beta}=\bigl\langle a^\gamma, \partial_{x_\alpha} a_\beta\bigr\rangle=-\bigl\langle \partial_{x_\alpha} a^\gamma,  a_\beta\bigr\rangle=\Gamma^\gamma_{\beta\alpha}.
$

Note that,
using the expansion $m=y_0+v$ and
$
(\nabla m)^T\nabla m=(\nabla y_0)^T\nabla y_0+(\nabla y_0)^T\nabla v+(\nabla v)^T\nabla y_0+\text{h.o.t}\in \mathbb{R}^{2\times 2},
$
the linear approximation of the difference  $\frac{1}{2}\big[{\rm I}_m - {\rm I}_{y_0}\big]^{\rm{lin}}$ appearing in the Koiter model  is easy to be obtained \cite[page 92]{Ciarlet00}, the linear approximation of the difference  $\big[{\rm II}_m - {\rm II}_{y_0}\big]^{\rm{lin}}$ needs some more insights from differential geometry \cite[page 95]{Ciarlet00} and it is based on 
{\it formulas of Gau\ss}  \ $\partial_{x_\alpha} a_\beta=\sum_{\gamma=1,2}\Gamma_{\alpha\beta}^\gamma \,a_\gamma+b_{\alpha\beta}a_3\qquad \text{and}\qquad   
\partial_{x_\beta}a^\alpha = - \sum_{\gamma=1,2}\Gamma^\alpha_{\gamma\beta}\,a^\gamma + b^\alpha_\beta\, n_0 
$
and the formulas of  {\it Weingarten}
$\partial_{x_\alpha}     a_3=\partial_{x_\alpha}   a^3$ \linebreak $= -\sum_{\beta=1,2}b_{\alpha\beta}\, a^\beta=-\sum_{\beta=1,2} b^\gamma_\beta\, a_\gamma\;
$
together with  the relations \cite[page 76]{Ciarlet00}
$  b_{\alpha\beta}(m)=-\bigl\langle\partial _\alpha a_3(m), a_\beta(m)\bigr\rangle=\bigl\langle\partial _\alpha a_\beta(m), a_3(m)\bigr\rangle=b_{\beta\alpha}(m),
$
where  $b_{\alpha\beta}(m)$ are the components of the second fundamental form corresponding to the map $m$, $b_{\alpha}^\beta(m)$ are the components of the matrix associated to the Weingarten map (shape operator), and on the following sets of  linear approximations of the normal to the surface
\begin{align}\label{liniarn}
n=\frac{\partial_{x_1} m\times \partial_{x_2} m}{\sqrt{\det ((\nabla m)^T\nabla m)}}=&\,\frac{1}{\sqrt{\det ((\nabla m)^T\nabla m)}}\left(\partial_{x_1} y_0\times \partial_{x_2} y_0+\partial_{x_1} y_0\times \partial_{x_2} v+\partial_{x_1} v\times \partial_{x_2} y_0+\text{h.o.t}\right),\notag\\
\det ((\nabla m)^T\nabla m)=&\,\det ((\nabla y_0)^T\nabla y_0)\, (1+ \tr(((\nabla y_0)^T\nabla y_0)^{-1}\, \sym ((\nabla y_0)^T\nabla v))+\text{h.o.t})\notag\\
=& \det ((\nabla y_0)^T\nabla y_0)(1+2\, \sum_{\alpha=1,2}\bigl\langle \partial_{x_\alpha}v, a^\alpha\bigr\rangle+\text{h.o.t})\notag\\
\frac{1}{\sqrt{\det ((\nabla m)^T\nabla m)}}=&\,\frac{1}{\sqrt{\det ((\nabla y_0)^T\nabla y_0)}}\, (1-\tr(((\nabla y_0)^T\nabla y_0)^{-1}\, \sym ((\nabla y_0)^T\nabla v))+\text{h.o.t}),\\
n=\frac{\partial_{x_1} m\times \partial_{x_2} m}{\sqrt{\det ((\nabla m)^T\nabla m)}}=&\,n_0+\frac{1}{\sqrt{\det ((\nabla y_0)^T\nabla y_0)}}\left(\partial_{x_1} y_0\times \partial_{x_2} v+\partial_{x_1} v\times \partial_{x_2} y_0+\text{h.o.t}\right)\notag\\
&\quad -\tr(((\nabla y_0)^T\nabla y_0)^{-1}\, \sym ((\nabla y_0)^T\nabla v) )\,n_0\notag.
\end{align}
Here, ``h.o.t" stands for terms of order higher than linear with respect to $v$.

We also note that other alternative (but equivalent) forms of the change of metric tensor and  the change of curvature tensor  given in \cite[page 181]{Ciarlet00} are
\begin{equation}\label{formK}
\mathcal{G}_{\rm{Koiter}}^{\rm{lin}} =  \Big( \frac{1}{2}(\partial_\beta v_\alpha+\partial_\alpha v_\beta)-\sum_{\gamma=1,2}\Gamma_{\alpha\beta}^\gamma v_\gamma-b_{\alpha\beta}v_3 \Big)_{\alpha\beta}\in {\rm Sym}(2),
\end{equation}
and \begin{align}\label{formR}
\mathcal{R}_{\rm{Koiter}}^{\rm{lin}} = & \Big( \partial_{x_\alpha x_\beta}v_3-\sum_{\gamma=1,2}\Gamma_{\alpha\beta}^\gamma \partial_{x_\gamma}v_3-\sum_{\gamma=1,2}b_\alpha^\gamma b_{\gamma\beta}v_3+\sum_{\gamma=1,2}b_{\alpha}^\gamma(\partial_{x_\beta}v_\gamma-\sum_{\tau=1,2}\Gamma_{\beta\gamma}^\tau v_\tau)\\&+\sum_{\gamma=1,2}b_{\beta}^\gamma(\partial_{x_\alpha}v_\gamma-\sum_{\tau=1,2}\Gamma_{\alpha\tau}^\gamma v_\gamma)
+\sum_{\tau=1,2}(\partial _{x_\alpha}b_\beta^\tau+\sum_{\gamma=1,2}\Gamma_{\alpha\gamma}^\tau b_\beta^\gamma-\sum_{\gamma=1,2}\Gamma_{\alpha\beta}^\gamma b_\gamma^\tau)v_\tau\Big)_{\alpha\beta}\in {{\rm Sym}(2)},\notag
\end{align}  respectively. Actually, on one hand, the last form \eqref{formR} of the curvature tensor will be considered when the admissible set of solutions of the variational problem will be defined.  On the other hand, as noticed in \cite[page 175]{Ciarlet2Diff-Geo2005} by considering the form \eqref{equ12} of the change of metric tensor,   we can impose substantially weaker   regularity assumptions  on the mapping $y_0$.

Regarding the existence of the solution, the following results are known,  {see \cite{blouza1999existence} and \cite{Ciarlet00}}.
\begin{theorem}{\rm (Existence results for the linear Koiter model on a general surface) \cite[Theorem 7.1.-1]{Ciarlet00}}\label{thexK}
	Let there be given a domain $\omega\subset \mathbb{R}^2$ and an injective mapping $y_0\in {\rm C}^{3}(\overline{\omega}, \mathbb{R}^3)$ such that the two vectors $a_\alpha=\partial_{x_\alpha}y_0$, $\alpha=1,2$, are linear independent at all points of $\overline{\omega}$.
	Then the variational problem  \eqref{Ap7matrix1} has one and only one solution  in the admissible set of solutions
	\begin{align}\label{admsetK1}
	\mathcal{A}_1:=\{v=(v_1,v_2,v_3)\in {\rm H}^1(\omega,\mathbb{R})\times{\rm H}^1(\omega,\mathbb{R})\times {\rm H}^2(\omega,\mathbb{R}) \,|\,  v_1=v_2=v_3=\bigl\langle \nabla v_3,\nu\bigr\rangle=0 \ \ \text{on}\ \ \partial \omega\},
	\end{align}
	where $\nu$ is the outer unit normal to $\partial \omega$.
\end{theorem}
\begin{theorem}{\rm (Existence results for the linear Koiter model for shells whose middle surfaces have little regularity)  {\cite{blouza1999existence}},\cite[Theorem 7.1.-2]{Ciarlet00}}\label{thexlr}
	Let there be given a domain $\omega\subset \mathbb{R}^2$ and an injective mapping $y_0\in {\rm H}^{2,\infty}(\omega, \mathbb{R}^3)$ such that the two vectors $a_\alpha=\partial_{x_\alpha}y_0$, $\alpha=1,2$,  are linear independent at all points of $\overline{\omega}$.
	Then the variational problem  \eqref{Ap7matrix1} has one and only one solution in the admissible set of solutions
	\begin{align}\label{admsetK2}
	\mathcal{A}_2:=\{v\in {\rm H}_0^1(\omega,\mathbb{R}^3)\,|\, \bigl\langle \partial_{x_\alpha x_\beta}v, n_0\bigr\rangle\in {\rm L}^2(\omega)\}
	\end{align}
\end{theorem}

The proof of Theorem \ref{thexK} is based on the next inequality of Korn's type on a general surface.
\begin{theorem}
	{\rm (Korn-type inequality  on a general surface)\label{Kornineq}\cite[Theorem 2.6.-4]{Ciarlet00}}\label{Kornlr}
	Let there be given a domain $\omega\subset \mathbb{R}^2$ and an injective mapping $y_0\in {\rm C}^{3}(\overline{\omega}, \mathbb{R}^3)$ such that the two vectors $a_\alpha=\partial_{x_\alpha}y_0$ are linear independent at all points of $\overline{\omega}$.
	Let  the change of metric and change of curvature be given by \eqref{formK} and \eqref{formR}, respectively. Then there exists a constant $c=c(\omega, \partial \omega, y_0)>0$ such that
	\begin{align}
	\lVert v_1\rVert_{{\rm H}^1(\omega;\mathbb{R})}^2+\lVert v_2\rVert_{{\rm H}^1(\omega;\mathbb{R})}^2+\lVert v_3\rVert_{{\rm H}^2(\omega;\mathbb{R})}^2\leq c\,\big[\lVert \mathcal{G}_{\rm{Koiter}}^{\rm{lin}} \rVert_{{\rm L}^2(\omega)}^2+\lVert \mathcal{R}_{\rm{Koiter}}^{\rm{lin}} \rVert_{{\rm L}^2(\omega)}^2\big]\qquad \forall \ v\in \mathcal{A}_1.
	\end{align}
\end{theorem}
The existence result given by Theorem \ref{thexlr} is based on another Korn-type inequality which is valid for  shells whose middle surfaces have little regularity.
\begin{theorem}{\rm (Korn inequality  on a general surface for shells whose middle surfaces have little regularity)  {\cite{blouza1999existence}}, \cite[Theorem 2.6.-6]{Ciarlet00}}\label{Kornlr2}
	Let there be given a domain $\omega\subset \mathbb{R}^2$ and an injective mapping $y_0\in {\rm H}^{2,\infty}(\omega, \mathbb{R}^3)$ such that the two vectors $a_\alpha=\partial_{x_\alpha}y_0$, $\alpha=1,2$,  are linear independent at all points of $\overline{\omega}$.
	Then there exists a constant $c>0$ such that \begin{align}
	\lVert v\rVert_{{\rm H}^1(\omega;\mathbb{R}^3)}^2+\sum_{\alpha,\beta=1,2}\bigl|\langle \partial_{x_\alpha x_\beta}v, n_0\bigr\rangle_{{\rm L}^2(\omega)}|^2\leq c\,\big[\lVert \mathcal{G}_{\rm{Koiter}}^{\rm{lin}} \rVert_{{\rm L}^2(\omega)}^2+\lVert \mathcal{R}_{\rm{Koiter}}^{\rm{lin}} \rVert_{{\rm L}^2(\omega)}^2\big]\qquad \forall \ v\in \mathcal{A}_2,
	\end{align}
	where   the change of metric and the change of curvature are given by \eqref{equ11} and \eqref{equ12}, respectively (note that 
	$
	\mathcal{G}_{\rm{Koiter}}^{\rm{lin}}\in{\rm L}^2(\omega) \  \text{and}\  \mathcal{R}_{\rm{Koiter}}^{\rm{lin}}\in{\rm L}^2(\omega),\ \ \forall \, v\in \mathcal{A}_2
	$).
\end{theorem}

\section{The linear constrained Cosserat-shell model}\setcounter{equation}{0}

 \subsection{The deformation measures in the linear constrained Cosserat-shell model}\label{Ccl}

We express the total midsurface deformation  as
\begin{align}
m(x_1,x_2)=y_0(x_1,x_2)+v(x_1,x_2),
\end{align}
with $v:\omega\to \mathbb{R}^3$, the infinitesimal shell-midsurface displacement. For the elastic rotation tensor $ \overline{Q}_{e,s}\in\rm{SO}(3) $ there is a skew-symmetric matrix  \begin{align}
\overline{A}_\vartheta:={\rm Anti}(\vartheta_1,\vartheta_2,\vartheta_3):=\begin{footnotesize}
\begin{pmatrix}
0&-\vartheta_3&\vartheta_2\\
\vartheta_3&0&-\vartheta_1\\
-\vartheta_2&\vartheta_1&0
\end{pmatrix}\end{footnotesize}\in \mathfrak{so}(3), \quad \qquad {\rm Anti}:\mathbb{R}^3\to \mathfrak{so}(3),
\end{align}
where $ \vartheta={\rm axl}( \overline{A}_\vartheta) $ denotes the axial vector of $ \overline{A}_\vartheta $, such that 
$ \overline{Q}_{e,s}:=\exp(\overline{A}_\vartheta)\;= \;\sum_{k=0}^{\infty} \frac{1}{k!} \,\overline{A}_\vartheta^k\; = \;\id_3 + \overline{A}_\vartheta+\textrm{h.o.t.}$

The tensor field $\overline{A}_\vartheta$ is   the infinitesimal elastic  microrotation. Here, ``h.o.t" stands for terms of order higher than linear with respect to $v$ and $\overline{A}_\vartheta$.  

Our aim now is to express all the deformation measures and the linearised models in terms of $v$ and $\overline{A}_\vartheta$, as well as in terms of its axial vector $\vartheta$. The following definitions are used to express  these quantities in terms of $\vartheta$.
For any column vector $ q\in \mathbb{R}^3$ and any  matrix $ M=(M_1|M_2|M_3)\in \mathbb{R}^{3\times 3}  $ we define the cross-product \cite{GhibaNeffPartIV}
\begin{align} \label{ec1} 
q\times M :=&\, (q\times M_1\,|\,q\times M_2\,|\,q\times M_3) \qquad\qquad(\mbox{operates on columns})\quad \textrm{and}
\vspace{6pt}\\
M^T \times q^T:=&\, - (q\times M)^T \qquad\qquad \qquad\qquad  \qquad(\mbox{operates on rows}).\notag
\end{align}
Note that $ M $ can also be a $ 3\times 2 $ matrix, the definition remains the same.
We note some properties of these operations: for any column vectors $ q_1, q_2\in \mathbb{R}^{3}  $ and any   matrices $ M , N \in \mathbb{R}^{3\times 3} $ (or $\in \mathbb{R}^{3\times 2} $) we have
\begin{align} \label{ec2} 
(q_1\times M )\,q_2 = &\, \,q_1\times (M \,    q_2)\,,
\qquad 
q_1^T(q_2\times M ) =  \,(q_1\times q_2)^T\, M =  (q_1^T\times  q_2^T)\,M =  -q_2^T(q_1\times M ) 
\end{align}
and, more general
\begin{align} \label{ec3} 
(q_1\times M )N = &\, q_1\,\times (M\, N)\,,
\qquad 
N^T(q_2\times M ) = \, -(q_2\times N)^T\, M  = (N^T\times  q_2^T)\,M\,.
\end{align}
With these relations, the infinitesimal microrotation $\overline{A}_\vartheta$ can be expressed as 
\begin{equation} \label{ec4} 
\overline{A}_\vartheta :=  \vartheta\times \id_3 = \id_3\times \vartheta^T\in \mathbb{R}^{3\times 3}.
\end{equation}

For $m=y_0+v$,  starting from 
\begin{align}\label{miul}
\mathcal{E}_{ \infty }
&=\sqrt{[\nabla\Theta \,]^{-T}\,{\rm I}_m^{\flat }\,[\nabla\Theta \,]^{-1}}-
\sqrt{[\nabla\Theta \,]^{-T}\,{\rm I}_{y_0}^{\flat }\,[\nabla\Theta \,]^{-1}}=\mathcal{E}_{ \infty }^{\rm{lin}}+\textrm{h.o.t.},
\end{align}
we find the expression of the linear approximation of $\mathcal{E}_{ \infty }$, given by
\begin{align}\label{miul2}
\ \,
[\nabla\Theta \,]^{-T}
\begin{footnotesize}\left( \begin{array}{c|c}
\mathcal{G}_{\infty}^{\rm{lin}} & 0 \vspace{4pt}\\
0  & 0
\end{array} \right)\end{footnotesize} [\nabla\Theta \,]^{-1}&=\mathcal{E}_{ \infty }^{\rm{lin}}=\frac{1}{2}\sqrt{[\nabla\Theta \,]^{-T}\,{\rm I}_{y_0}^{\flat }\,[\nabla\Theta \,]^{-1}}^{-1}\,[\nabla\Theta \,]^{-T}\left( (\nabla y_0)^T\nabla v+(\nabla v)^T\nabla y_0\right)^\flat[\nabla\Theta \,]^{-1}\notag
\\
&=\frac{1}{2}[\nabla\Theta \,]^{-T}\left( (\nabla y_0)^T\nabla v+(\nabla v)^T\nabla y_0\right)^\flat[\nabla\Theta \,]^{-1}\\&=[\nabla\Theta \,]^{-T}[\underbrace{\sym ((\nabla y_0)^T\nabla v)}_{:=\,\mathcal{G}_{ \infty }^{\rm lin}}]^\flat[\nabla\Theta \,]^{-1} ,\notag
\end{align}
where $\mathcal{G}_{\infty}^{\rm{lin}} $ denotes the linear approximation of $\mathcal{G}_{\infty}$ and we have used the expansion  $\sqrt{x+\delta x}-\sqrt{\delta x}=\dfrac{x}{2 \,\delta x}+O(x^2)\quad  \forall \, x\geq 0, \delta x>0.$

While we have the identity
\begin{equation}
\label{equ12,5c}
\mathcal{G}_\infty = \mathcal{G}_{\rm{Koiter}} \in {\rm Sym}(2) \qquad \text{and}\qquad \mathcal{G}^{\rm{lin}}_\infty = \mathcal{G}_{\rm{Koiter}}^{\rm{lin}} \in {\rm Sym}(2),
\end{equation} 
 we are only able to find the  relation 
 \begin{equation}
 \label{equ220} 
 \mathcal{R}^{\rm{lin}}_\infty\,  = \,\underbrace{ \mathcal{R}_{\rm{Koiter}}^{\rm{lin}}}_{\in {\rm Sym}(2)} - \,\underbrace{\mathcal{G}_{\rm{Koiter}}^{\rm{lin}}\,{\rm L}_{y_0}}_{\not\in {\rm Sym}(2)} \not\in {\rm Sym}(2)
 \end{equation}  between $ \mathcal{R}^{\rm{lin}}_\infty$ and $\; \mathcal{R}_{\rm{Koiter}}^{\rm{lin}} $ and not between  $ \mathcal{R}_\infty$ and $\mathcal{R}_{\rm{Koiter}} $. This is not surprising, since only symmetric stress tensors are taken into account in the classical linear Koiter model, i.e., the internal strain energy does not depend on the skew-symmetric part of the considered strain measures (since it is work conjugate to the skew-symmetric part of the stress tensor). In addition,   the linear Koiter model does not consider  extra degrees of freedom.

In order to see how the rotation vector $\vartheta$ depends on $v$ in the linear constrained Cosserat-shell model, let us do something different, and linearise the condition $\mathcal{G}_{\infty} :=\, (\overline{Q}_{\infty} \nabla y_0)^{T} \nabla m- {\rm I}_{y_0}\in {\rm Sym}(2)$  to obtain 
\begin{align}\label{equ13l}
(\nabla y_0)^T(\id-\overline{A}_{\vartheta_\infty} )(\nabla y_0+\nabla v)+\textrm{h.o.t.}=&\,(\nabla y_0+\nabla v)^T(\id+\overline{A}_{\vartheta_\infty} )\, (\nabla y_0)+\textrm{h.o.t.}\qquad \textrm{and}\qquad\\ 
(\id+\overline{A}_{\vartheta_\infty} )  n_0+\textrm{h.o.t.} =&\,n,\notag
\end{align}
i.e.
\begin{align}\label{equ13l2}
(\nabla y_0)^T(\nabla v)-(\nabla y_0)^T\overline{A}_{\vartheta_\infty} \,\nabla y_0+\textrm{h.o.t.}=&\,(\nabla v)^T (\nabla y_0)+(\nabla y_0)^T\overline{A}_{\vartheta_\infty}  (\nabla y_0)+\textrm{h.o.t.}\qquad \textrm{and}\qquad\\ 
n=&\,n_0+\vartheta_\infty   \times n_0+\textrm{h.o.t.}\notag
\end{align}
or alternatively 
\begin{align}\label{equ13l4}
\skw[(\nabla y_0)^T(\nabla v)]=&\,(\nabla y_0)^T[\vartheta_\infty  \times (\nabla y_0)]+\textrm{h.o.t.}\qquad \textrm{and}\\ n=&\,n_0+\vartheta_\infty   \times n_0+\textrm{h.o.t.}\notag
\end{align}

Using that $\mathcal{T}_\infty^{\rm lin}+{\rm h.o.t.}=\mathcal{T}_\infty=0$, i.e. $ \, (\overline{Q}_{\infty}  n_0)^T (\nabla m)=0 $, and that the linear approximation of the {\it transverse shear vector} $\mathcal{T}_\infty=(\overline{Q}_{\infty}  n_0)^T (\nabla m)$ is given by \cite{GhibaNeffPartIV} \begin{align}\label{ec13}\mathcal{T}^{\rm{lin}}_\infty=  n_0^T ( \nabla v - \overline{A}_{\vartheta_\infty}\nabla y_0)& \qquad\Leftrightarrow \qquad \mathcal{T}^{\rm{lin}}_\infty =  n_0^{T} ( \nabla v - \vartheta_\infty\times\nabla y_0 )  \notag
\\&\qquad\Leftrightarrow \qquad \mathcal{T}^{\rm{lin}}_\infty  = n_0^{T}  \nabla v + \vartheta^T_\infty(n_0\times\nabla y_0 ) \;=\;  
n_0^{T}  \nabla v + (\vartheta_\infty\times n_0)^T\nabla y_0 \;,
\end{align} we find
\begin{align} \label{equ17}  
 (\vartheta_\infty \times n_0)^T\nabla y_0=-n_0^{T}  \nabla v  \;\  &\Leftrightarrow \  \bigl\langle a_\alpha,\vartheta_\infty \times n_0\bigr\rangle = - \bigl\langle n_0, \partial_{x_\alpha}v\bigr\rangle\;\  \Leftrightarrow \
\vartheta_\infty \times n_0 = -\dd\sum_{\alpha=1,2}\bigl\langle n_0, \partial_{x_\alpha}v\bigr\rangle\, a^\alpha.
\end{align}
Hence, we obtain finally
\begin{align}
&n=n_0-\dd\sum_{\alpha=1,2}\bigl\langle n_0, \partial_{x_\alpha}v\bigr\rangle\, a^\alpha+\textrm{h.o.t.}
\end{align}

We can also express the $ 2\times 2 $ skew-symmetric matrix $ {\rm C}=  \sqrt{\det {\rm I}_{y_0}}\begin{footnotesize}\begin{pmatrix}
	0 & 1
	\\
	-1  & 0
	\end{pmatrix}\end{footnotesize}   $ using the cross product. 
	Similar as in \cite{GhibaNeffPartIV}, we have 
	\begin{equation} \label{ec8} 
	(a^1\,|\,a^2)\; {\rm C}  \; = - n_0\times \nabla y_0
	\,.
	\end{equation}
	Thus, we deduce
	\begin{equation} \label{ec9} 
	{\rm C}  \; = - (\nabla y_0)^T(n_0\times \nabla y_0) = (n_0\times \nabla y_0)^T(\nabla y_0)\,,
	\end{equation}
	and  it holds
	\begin{equation} \label{ec9,5} 
	\bigl\langle \vartheta_\infty ,n_0\bigr\rangle\,{\rm C}  \; = - (\nabla y_0)^T(\vartheta_\infty \times \nabla y_0) = (\vartheta_\infty \times \nabla y_0)^T(\nabla y_0)\,.
	\end{equation}
	
 Let us now obtain an alternative form of $\mathcal{G}^{\rm lin}_\infty$.
	We have
	\begin{equation}\label{equ3}
	\begin{array}{rcl}
	\mathcal{G}^{\rm lin}_\infty &=&    (\nabla y_0)^{T} ( \nabla v - \overline{A}_{\vartheta_\infty}\nabla y_0 ) \;=\; (\nabla y_0)^{T} ( \partial_{x_1} u + a_1\times \vartheta_\infty\;|\;  \partial_{x_2} u + a_2\times \vartheta_\infty) 
	\vspace{6pt}\\
	&=&  (\nabla y_0)^{T}(\nabla v) + (a_1\,|\,a_2)^T(  a_1\times \vartheta_\infty\;|\;    a_2\times \vartheta_\infty)\vspace{1mm}= (\nabla y_0)^{T}(\nabla v) + \begin{footnotesize}
	\begin{pmatrix}
	0 & \bigl\langle \vartheta_\infty, a_1\times a_2\bigr\rangle
	\\
	-\bigl\langle \vartheta_\infty, a_1\times a_2\bigr\rangle & 0
	\end{pmatrix}\end{footnotesize}]
	\vspace{6pt}\\
	&=& (\nabla y_0)^{T}(\nabla v) + \bigl\langle\vartheta_\infty, n_0\bigr\rangle\begin{footnotesize}
	\begin{pmatrix}
	0 & \sqrt{\det {\rm I}_{y_0}}
	\\
	-\sqrt{\det {\rm I}_{y_0}}  & 0
	\end{pmatrix}\end{footnotesize} =   (\nabla y_0)^{T}(\nabla v) + \bigl\langle\vartheta_\infty, n_0\bigr\rangle\, {\rm C}.
	\end{array}
	\end{equation}
	Remember  that ${\rm C}$  is 
	a $ 2\times 2  $ skew-symmetric matrix. 
Using in the following \eqref{equ3}  and the condition $\mathcal{G}^{\rm{lin}}_
\infty\in {\rm Sym}(2)$ (i.e. $ \mbox{skew}\, (\mathcal{G}^{\rm{lin}}_\infty) =0 $), we have
\begin{align}
\label{equ18} 
\bigl\langle\vartheta_\infty ,n_0\bigr\rangle\,{\rm C} = &\, -\mbox{skew}\,\big[ (\nabla y_0)^{T}(\nabla v)\big]\quad \Leftrightarrow \quad \bigl\langle\vartheta_\infty ,n_0\bigr\rangle\id_2 = -\mbox{skew}\,\big[ (\nabla y_0)^{T}(\nabla v)\big]\,{\rm C}^{-1}\\\quad \Leftrightarrow \quad 2\,\bigl \langle\vartheta_\infty &,n_0\bigr\rangle = \, -\tr(\mbox{skew}\,\big[ (\nabla y_0)^{T}(\nabla v)\big]\,{\rm C}^{-1})\quad \Leftrightarrow \quad \bigl \langle\vartheta_\infty ,n_0\bigr\rangle = \, -\frac{1}{2}\,\tr(\mbox{skew}\,\big[ (\nabla y_0)^{T}(\nabla v)\big]\,{\rm C}^{-1}).\notag
\end{align}
Thus, the rotation vector $ \vartheta_\infty  $ is completely determined by the relations \eqref{equ17} and \eqref{equ18} as a function of the midsurface displacement $ v $, by knowing its component in the direction of $n_0$ and its components in the tangent plane.

Further,  we can write
\begin{equation}
\label{equ19} 
\mathcal{R}^{\rm{lin}}_\infty  = 
(\nabla y_0)^{T} ( \vartheta_\infty \times \partial_{x_1}n_0\;|\; \vartheta_\infty \times \partial_{x_1}n_0) - (\nabla y_0)^{T} \big( \partial_{x_1}(\vartheta_\infty \times n_0)\;|\; \partial_{x_2}(\vartheta_\infty \times n_0) \big).
\end{equation}
To compute the last term in \eqref{equ19} we use \eqref{equ17} and write
\begin{align} 
-\partial_{x_\beta}(\vartheta_\infty \times n_0) =&\,  \partial_{x_\beta}\big[\dd\sum_{\alpha=1,2}\bigl\langle n_0, \partial_{x_\alpha}v\bigr\rangle a^\alpha\big] \;=\; \dd\sum_{\alpha=1,2}\partial_{x_\beta}[\bigl\langle n_0,\partial_{x_\alpha} v\bigr\rangle\, a^\alpha]  + \dd\sum_{\alpha=1,2} [ \bigl\langle n_0, \partial_{x_\alpha}v\bigr\rangle\partial_{x_\beta}a^\alpha
]\vspace{6pt}\\
= & \dd\sum_{\alpha=1,2}[\bigl\langle n_0, \partial_{x_\alpha x_\beta}v\bigr\rangle\, a^\alpha] +\dd\sum_{\alpha=1,2} [\bigl\langle\partial_{x_\beta} n_0, \partial_{x_\alpha} v\bigr\rangle a^\alpha] + \dd\sum_{\alpha=1,2}[\bigl\langle n_0, \partial_{x_\alpha}v\bigr\rangle ( -\dd\sum_{\gamma=1,2} \Gamma^\alpha_{\gamma\beta}\,a^\gamma + b^\alpha_\beta\, n_0)]
\vspace{6pt}\notag\\
= &\dd\sum_{\alpha=1,2}[\bigl\langle n_0,\partial_{x_\alpha x_\beta}v - \dd\sum_{\gamma=1,2}\Gamma^\gamma_{\alpha\beta}\partial_{x_\gamma} v\bigr\rangle a^\alpha] - \dd\sum_{\alpha,\gamma=1,2}[\bigl\langle a_\gamma, \partial_{x_\alpha}v\bigr\rangle b^\gamma_\beta\, a^\alpha] + \dd\sum_{\alpha=1,2}[\bigl\langle n_0, \partial_{x_\alpha}v\bigr\rangle\, b^\alpha_\beta\, n_0]\;,\notag
\end{align}
due to Gau\ss  \ and Weingarten formulas \cite{DoCarmo}. Therefore, we obtain 
\begin{equation} \label{equ20} 
\begin{array}{rcl}
- (\nabla y_0)^{T} \big( \partial_{x_1}(\vartheta_\infty \times n_0)\;|\; \partial_{x_2}(\vartheta_\infty \times n_0) \big) &=&   
\mathcal{R}_{\rm{Koiter}}^{\rm{lin}} - (\partial_{x_1} v\,|\, \partial_{x_2}v)^T\, (a_1\,|\, a_2)
\begin{footnotesize}\begin{pmatrix}
b^1_1  & b^1_2
\vspace{4pt}\\
b^2_1  & b^2_2
\end{pmatrix}\end{footnotesize}
\vspace{6pt}\\
&= &  \mathcal{R}_{\rm{Koiter}}^{\rm{lin}} - \big[ (\nabla y_0)^{T}(\nabla v)\big]^T\,{\rm L}_{y_0}\;.
\end{array}
\end{equation}
To compute the first term in \eqref{equ19} we decompose $ \; \vartheta_\infty = \bigl\langle \vartheta_\infty , n_0\bigr\rangle\, n_0 + \dd\sum_{\alpha=1,2}[\bigl\langle\vartheta_\infty , a^\alpha\bigr\rangle\, a_\alpha]\; $ and write
\begin{equation} 
\vartheta_\infty \times\partial_{x_\beta} n_0= \bigl\langle\vartheta_\infty , n_0\bigr\rangle n_0\times \partial_{x_\beta}n_0  + \dd\sum_{\alpha=1,2}[\bigl\langle\vartheta_\infty , a^\alpha\bigr\rangle a_\alpha\times \partial_{x_\beta}n_0] \;.
\end{equation}
Using
\begin{equation} n_0\times \partial_{x_\beta} n_0 = n_0\times (-\sum_{\alpha=1,2} b^\alpha_\beta\,a_\alpha) = \,\sqrt{\det {\rm I}_{y_0}} \;(-b^1_\beta\,a^2 + b^2_\beta\,a^1) \end{equation}
and
\begin{equation} a_\alpha\times \partial_{x_\beta}n_0 = a_\alpha\times (-\sum_{\gamma=1,2} b^\gamma_\beta\,a_\gamma) = \, -\sum_{\gamma=1,2} b^\gamma_\beta\,\epsilon_{\alpha\gamma}\,n_0\;, \, \end{equation}
with $\epsilon_{12}=-\epsilon_{21}=1$, this yields
\begin{equation}
\begin{array}{rcl}
(\nabla y_0)^{T} ( \vartheta_\infty \times \partial_{x_1}n_0\;|\; \vartheta_\infty \times \partial_{x_2}n_0 )  &=&   
\sqrt{\det {\rm I}_{y_0}} \;\bigl\langle \vartheta_\infty , n_0\bigr\rangle \,(\nabla y_0)^{T} \,(-b^1_1\,a^2 + b^2_1\,a^1\;|\; -b^1_2\,a^2 + b^2_2\,a^1 ) 
\vspace{6pt}\\
&= &   \sqrt{\det {\rm I}_{y_0}} \;\bigl\langle \vartheta_\infty , n_0\bigr\rangle \,(\nabla y_0)^{T} \,(- a^2  \,|\,  a^1 ) \begin{footnotesize}\begin{pmatrix}
b^1_1  & b^1_2
\vspace{4pt}\\
b^2_1  & b^2_2
\end{pmatrix}\end{footnotesize}
\vspace{6pt}\\
&= &   \bigl\langle \vartheta_\infty , n_0\bigr\rangle \,(\nabla y_0)^{T} \,(a^1  \,|\,  a^2 )\,{\rm C}\,{\rm L}_{y_0} \;=\;  \bigl\langle \vartheta_\infty , n_0\bigr\rangle \,{\rm C}\,{\rm L}_{y_0}
\end{array}
\end{equation}
and inserting here \eqref{equ18} we obtain
\begin{equation}
\label{equ21} 
(\nabla y_0)^{T} ( \vartheta_\infty \times \partial_{x_1}n_0\;|\; \vartheta_\infty \times \partial_{x_2} n_0) \; =\; 
-\mbox{skew}\,\big[ (\nabla y_0)^{T}(\nabla v)\big]\,{\rm L}_{y_0}
.
\end{equation}
Substituting \eqref{equ20} and \eqref{equ21} into \eqref{equ19}, we are left with
\begin{align}
\label{equ21,5} 
\mathcal{R}^{\rm{lin}}_\infty &= \,\, \mathcal{R}_{\rm{Koiter}}^{\rm{lin}} - \mbox{sym}\,\big[ (\nabla y_0)^{T}(\nabla v)\big]\,{\rm L}_{y_0},\qquad \qquad 
\mathcal{G}^{\rm lin}_\infty=\, \mbox{sym}\,\big[ (\nabla y_0)^{T}(\nabla v)\big]=\,\mathcal{G}_{\rm{Koiter}}^{\rm lin},
\end{align}
which can be written by virtue of \eqref{equ12,5c} in the form
\begin{equation}
\label{equ22} 
\mathcal{R}^{\rm{lin}}_\infty\,  = \,\, \mathcal{R}_{\rm{Koiter}}^{\rm{lin}} - \,\mathcal{G}_{\rm{Koiter}}^{\rm{lin}} \,{\rm L}_{y_0}.
\end{equation}
Some direct calculations are also possible, see also \eqref{equ13l2},  since 
\begin{align}
\overline{Q}_\infty=&\,{\rm polar}((\nabla m\,|\, n)[\nabla \Theta]^{-1})={\rm polar}([(\nabla y_0+\nabla v\,|\, n_0+\vartheta_\infty   \times n_0)[\nabla \Theta]^{-1})\notag\\=&\,{\rm polar}([\nabla \Theta+ (\nabla v\,|\,\delta n)[\nabla \Theta]^{-1})=
{\rm polar}(\id_3+ (\nabla v\,|\vartheta_\infty   \times n_0)[\nabla \Theta]^{-1})\\=&
\,\id_3+\skw( (\nabla v\,|\vartheta_\infty   \times n_0)[\nabla \Theta]^{-1})+{\rm h.o.t.},\notag
\\
=&
\,\id_3-\skw( (\nabla v\,|\dd\sum_{\alpha=1,2}\bigl\langle n_0, \partial_{x_\alpha}v\bigr\rangle\, a^\alpha)[\nabla \Theta]^{-1})+{\rm h.o.t.}\notag
\end{align}
which leads to  
\begin{align}
\overline{A}_{\vartheta_\infty}=-\skw( (\nabla v\,|\dd\sum_{\alpha=1,2}\bigl\langle n_0, \partial_{x_\alpha}v\bigr\rangle\, a^\alpha)[\nabla \Theta]^{-1}).
\end{align}
Summarizing, we have the linear approximations of the deformations measures 
\begingroup
\allowdisplaybreaks
\begin{align}\label{eq12l2}
\mathcal{E}_{ \infty }^{\rm{lin}}  = &\quad \  [\nabla\Theta \,]^{-T}
[\mathcal{G}_{\rm{Koiter}}^{\rm{lin}}]^\flat [\nabla\Theta \,]^{-1},\notag\\
\mathrm{C}_{y_0} \mathcal{K}_{ \infty }^{\rm{lin}} = &\,- [\nabla\Theta \,]^{-T}\big[\mathcal{R}_{\rm{Koiter}}^{\rm{lin}} - \,\mathcal{G}_{\rm{Koiter}}^{\rm{lin}} \,{\rm L}_{y_0}\big]^\flat [\nabla\Theta \,]^{-1}\notag\\
\mathrm{C}_{y_0} \mathcal{K}_{ \infty }^{\rm{lin}}{\rm B}_{y_0} = &\, - [\nabla\Theta \,]^{-T}
[
(\mathcal{R}_{\rm{Koiter}}^{\rm{lin}} - \,\mathcal{G}_{\rm{Koiter}}^{\rm{lin}} \,{\rm L}_{y_0}) {\rm L}_{y_0}]^\flat [\nabla\Theta \,]^{-1},
\vspace{10pt}\\
\mathcal{E}_{ \infty }^{\rm{lin}} {\rm B}_{y_0}  + \mathrm{C}_{y_0} \mathcal{K}_{ \infty }^{\rm{lin}} 
= & -\,[\nabla\Theta \,]^{-T} [\mathcal{R}_{\rm{Koiter}}^{\rm{lin}}-\mathbf{2}\,\mathcal{G}_{\rm{Koiter}}^{\rm{lin}} \,{\rm L}_{y_0} ]^\flat
[\nabla\Theta \,]^{-1}\notag,\\
\mathcal{E}_{ \infty }^{\rm{lin}} {\rm B}_{y_0}^2  + \mathrm{C}_{y_0} \mathcal{K}_{ \infty }^{\rm{lin}} {\rm B}_{y_0}
= &\, -[\nabla\Theta \,]^{-T}
[(\mathcal{R}_{\rm{Koiter}}^{\rm{lin}}-\mathbf{2}\,\mathcal{G}_{\rm{Koiter}}^{\rm{lin}} \,{\rm L}_{y_0}){\rm L}_{y_0}]^\flat [\nabla\Theta \,]^{-1}\notag
,\\
\overline{A}_{\vartheta_\infty}&\equiv{\rm Anti}\vartheta_\infty=-\skw( (\nabla v\,|\dd\sum_{\alpha=1,2}\bigl\langle n_0, \partial_{x_\alpha}v\bigr\rangle\, a^\alpha)[\nabla \Theta]^{-1}),\notag\\
\vartheta_\infty &=  -\frac{1}{2}\,\tr(\mbox{skew}\,\big[ (\nabla y_0)^{T}(\nabla v)\big]\,{\rm C}^{-1})\,n_0-\sum_{\alpha=1,2}\bigl\langle n_0, \partial_{x_\alpha}v\bigr\rangle\, a^\alpha\in\mathbb{R}^3,\notag\\ \mathcal{K}_\infty^{\rm{lin}} \,&= \quad \  (\nabla\vartheta_\infty \, |\, 0) \; [\nabla\Theta \,]^{-1},\notag \\ \mathcal{K}_\infty^{\rm{lin}} {\rm B}_{y_0}\,&= \quad\  (\nabla\vartheta_\infty \, |\, 0) \; {\rm L}_{y_0}^\flat[\nabla\Theta \,]^{-1}, \notag\\ \mathcal{K}_\infty^{\rm{lin}} {\rm B}_{y_0}^2\,&= \quad \ (\nabla\vartheta_\infty \, |\, 0) \; ({\rm L}_{y_0}^\flat)^2[\nabla\Theta \,]^{-1}\notag.\notag
\end{align}
\endgroup
\subsection{The constrained linear $O(h^5)$-Cosserat shell model. Conditional existence }
We can now easily obtain the linearised constrained Cosserat $O(h^5)$-shell model  by inserting the linearised deformation measures in the quadratic variational problem of the nonlinear constrained Cosserat model and the obtained  minimization problem is  to find the midsurface displacement vector field
$v:\omega\subset\mathbb{R}^2\to\mathbb{R}^3$  minimizing on $\omega$:
\begin{align}\label{minvarlccond}
I= \int_{\omega}   \,\, \Big[ & \Big(h+{\rm K}\,\dfrac{h^3}{12}\Big)\,
W_{{\rm shell}}^{\infty}\big([\nabla\Theta \,]^{-T}
[\mathcal{G}_{\rm{Koiter}}^{\rm{lin}}]^\flat [\nabla\Theta \,]^{-1} \big)\vspace{2.5mm}\notag\\&+   \Big(\dfrac{h^3}{12}\,-{\rm K}\,\dfrac{h^5}{80}\Big)\,
W_{{\rm shell}}^{\infty}  \big( [\nabla\Theta \,]^{-T} [\mathcal{R}_{\rm{Koiter}}^{\rm{lin}}-\mathbf{2}\,\mathcal{G}_{\rm{Koiter}}^{\rm{lin}} \,{\rm L}_{y_0} ]^\flat
[\nabla\Theta \,]^{-1}\big) 
\vspace{2.5mm}\notag\\&+\dfrac{h^3}{3} \mathrm{ H}\,\mathcal{W}_{{\rm shell}}^{\infty}  \big(  [\nabla\Theta \,]^{-T}
[\mathcal{G}_{\rm{Koiter}}^{\rm{lin}}]^\flat [\nabla\Theta \,]^{-1} , [\nabla\Theta \,]^{-T} [\mathcal{R}_{\rm{Koiter}}^{\rm{lin}}-\mathbf{2}\,\mathcal{G}_{\rm{Koiter}}^{\rm{lin}} \,{\rm L}_{y_0} ]^\flat
[\nabla\Theta \,]^{-1} \big)\notag\\&-
\dfrac{h^3}{6}\, \mathcal{W}_{{\rm shell}}^{\infty}  \big(  [\nabla\Theta \,]^{-T}
[\mathcal{G}_{\rm{Koiter}}^{\rm{lin}}]^\flat [\nabla\Theta \,]^{-1} , [\nabla\Theta \,]^{-T}
[(\mathcal{R}_{\rm{Koiter}}^{\rm{lin}}-\mathbf{2}\,\mathcal{G}_{\rm{Koiter}}^{\rm{lin}} \,{\rm L}_{y_0})\,{\rm L}_{y_0}]^\flat [\nabla\Theta \,]^{-1}\big)\vspace{2.5mm}\notag\\&+ \,\dfrac{h^5}{80}\,\,
W_{\mathrm{mp}}^{\infty} \big([\nabla\Theta \,]^{-T}
[(\mathcal{R}_{\rm{Koiter}}^{\rm{lin}}-\mathbf{2}\,\mathcal{G}_{\rm{Koiter}}^{\rm{lin}} \,{\rm L}_{y_0})\,{\rm L}_{y_0} ]^\flat [\nabla\Theta \,]^{-1}\big)\vspace{2.5mm}\\&+ \Big(h-{\rm K}\,\dfrac{h^3}{12}\Big)\,
W_{\mathrm{curv}}\big( (\nabla\vartheta_\infty \, |\, 0) \; [\nabla\Theta \,]^{-1} \big)    \vspace{2.5mm}\notag\\&+  \Big(\dfrac{h^3}{12}\,-{\rm K}\,\dfrac{h^5}{80}\Big)\,
W_{\mathrm{curv}}\big(  (\nabla\vartheta_\infty \, |\, 0) \; {\rm L}_{y_0}^\flat[\nabla\Theta \,]^{-1} \big)  \vspace{2.5mm}\notag\\&+ \,\dfrac{h^5}{80}\,\,
W_{\mathrm{curv}}\big(   (\nabla\vartheta_\infty \, |\, 0) \; ({\rm L}_{y_0}^\flat)^2[\nabla\Theta \,]^{-1}  \big)\notag
\Big] \,{\rm det}(\nabla y_0|n_0)       \,\mathrm{\rm da}- \overline{\Pi}_\infty(v),\notag
\end{align}
such that
\begin{align}\label{contsyml2}
\mathcal{G}_{\rm{Koiter}}^{\rm{lin}} \,{\rm L}_{y_0} &\in{\rm Sym}(2) \qquad \textrm{and}\qquad 
(\mathcal{R}_{\rm{Koiter}}^{\rm{lin}}-2\,\mathcal{G}_{\rm{Koiter}}^{\rm{lin}} \,{\rm L}_{y_0})\,{\rm L}_{y_0}\in{\rm Sym}(2),
\end{align}
where 
\begingroup\allowdisplaybreaks
\begin{align}  \mathcal{G}_{\rm{Koiter}}^{\rm{lin}}&=\mbox{sym}\,\big[ (\nabla y_0)^{T}(\nabla v)\big]\notag\\
&=\Big( \frac{1}{2}(\partial_\beta v_\alpha+\partial_\alpha v_\beta)-\sum_{\gamma=1,2}\Gamma_{\alpha\beta}^\gamma v_\gamma-b_{\alpha\beta}v_3 \Big)_{\alpha\beta}\notag\\
 \mathcal{R}_{\rm{Koiter}}^{\rm{lin}}&=  \big( \bigl\langle n_0 ,  \partial_{x_\alpha  x_\beta}\,v- \dd\sum_{\gamma=1,2}\Gamma^\gamma_{\alpha\beta}\,\partial_{x_\gamma}\,v\bigr\rangle\big)_{\alpha\beta}\notag\\&=  \Big( \partial_{x_\alpha x_\beta}v_3-\sum_{\gamma=1,2}\Gamma_{\alpha\beta}^\gamma \partial_{x_\gamma}v_3-\sum_{\gamma=1,2}b_\alpha^\gamma b_{\gamma\beta}v_3+\sum_{\gamma=1,2}b_{\alpha}^\gamma(\partial_{x_\beta}v_\gamma-\sum_{\tau=1,2}\Gamma_{\beta\gamma}^\tau v_\tau)\notag\\&\ \ +\sum_{\gamma=1,2}b_{\beta}^\gamma(\partial_{x_\alpha}v_\gamma-\sum_{\tau=1,2}\Gamma_{\alpha\tau}^\gamma v_\gamma)
 +\sum_{\tau=1,2}(\partial _{x_\alpha}b_\beta^\tau+\sum_{\gamma=1,2}\Gamma_{\alpha\gamma}^\tau b_\beta^\gamma-\sum_{\gamma=1,2}\Gamma_{\alpha\beta}^\gamma b_\gamma^\tau)v_\tau\Big)_{\alpha\beta},
\\
W_{{\rm shell}}^{\infty}(  S)  &=   \mu\,\lVert\,   S\rVert^2  +\,\dfrac{\lambda\,\mu}{\lambda+2\mu}\,\big[ \mathrm{tr}   \, (S)\big]^2,\qquad \qquad \qquad 
\mathcal{W}_{{\rm shell}}^{\infty}(  S,  T) =   \mu\,\bigl\langle  S,   T\bigr\rangle+\,\dfrac{\lambda\,\mu}{\lambda+2\mu}\,\mathrm{tr}  (S)\,\mathrm{tr}  (T), \notag\vspace{2.5mm}\\
W_{\mathrm{mp}}^{\infty}(  S)&= \mu\,\lVert  S\rVert^2+\,\dfrac{\lambda}{2}\,\big[ \mathrm{tr}\,   (S)\big]^2 \quad \forall \ S,T\in{\rm Sym}(3), \notag\vspace{2.5mm}\\
W_{\mathrm{curv}}(  X )&=\mu\,{\rm L}_c^2\left( b_1\,\lVert \dev\,\sym \, X\rVert^2+b_2\,\lVert\skw \,X\rVert^2+b_3\,
[\tr(X)]^2\right) \quad \forall\  X\in\mathbb{R}^{3\times 3}\notag
\end{align}
and $\overline{\Pi}_\infty(v)$ is  the linearization of $\overline{\Pi}(m,\overline{Q}_{e,s})$ expressed only in terms of $v$.
\endgroup

We define the following linear  operators (we use the same notations, but we point out the dependence on $v$ of all involved quantities)
\begin{align}\label{operatorscon1}
\mathcal{G}_{\rm{Koiter}}^{\rm{lin}}&:{\rm H}^1_0(\omega, \mathbb{R}^3) \to \mathbb{R}^{2\times 2}, \qquad \qquad  \mathcal{G}_{\rm{Koiter}}^{\rm{lin}}(v)=\mbox{sym}\,\big[ (\nabla y_0)^{T}(\nabla v)\big]\notag\\
\mathcal{R}_{\rm{Koiter}}^{\rm{lin}}&:{\rm H}^2(\omega, \mathbb{R}^3) \to \mathbb{R}^{2\times 2},\qquad\qquad  \mathcal{R}_{\rm{Koiter}}^{\rm{lin}}(v)=  \big( \bigl\langle n_0 ,  \partial_{x_\alpha  x_\beta}\,v- \dd\sum_{\gamma=1,2}\Gamma^\gamma_{\alpha\beta}\,\partial_{x_\gamma}\,v\bigr\rangle\big)_{2\times 2}.
\end{align}

Moreover, we define two sets   of admissible functions 
\begin{align}\label{21l}
\mathcal{A}_{\rm lin}^{\infty}=\Bigg\{v=(v_1,v_2,v_3)&\in {\rm H}^1(\omega,\mathbb{R})\times{\rm H}^1(\omega,\mathbb{R})\times {\rm H}^2(\omega,\mathbb{R}) \,\big|\,  v_1=v_2=v_3=\bigl\langle \nabla v_3,\nu\bigr\rangle=0 \ \ \text{on}\ \ \partial \omega,\notag\\&\nabla \underbrace{[\skw( (\nabla v\,|\dd\sum_{\alpha=1,2}\bigl\langle n_0, \partial_{x_\alpha}v\bigr\rangle\, a^\alpha)[\nabla \Theta]^{-1})]}_{=\, \vartheta_\infty}\in {\rm L}^2(\omega),\\&\mathcal{G}_{\rm{Koiter}}^{\rm{lin}}(v) \,{\rm L}_{y_0} \in{\rm \textbf{Sym}}(2) \quad \textrm{and}\quad 
(\mathcal{R}_{\rm{Koiter}}^{\rm{lin}}(v)-\mathbf{2}\,\mathcal{G}_{\rm{Koiter}}^{\rm{lin}}(v) \,{\rm L}_{y_0})\,{\rm L}_{y_0}\in{\rm \textbf{Sym}}(2)
\Bigg\},\notag
\end{align}
where  $\nu$ is the outer unit normal to $\partial \omega$,
and
\begin{align}\label{21l2}
\widehat{\mathcal{A}}_{\rm lin}^{\infty}=\Bigg\{v\in {\rm H}_0^1(\omega,\mathbb{R}^3)&\,|\, \bigl\langle \partial_{x_\alpha x_\beta}v, n_0\bigr\rangle\in {\rm L}^2(\omega),\ \ \ \nabla [\skw( (\nabla v\,|\dd\sum_{\alpha=1,2}\bigl\langle n_0, \partial_{x_\alpha}v\bigr\rangle\, a^\alpha)[\nabla \Theta]^{-1})]\in {\rm L}^2(\omega),\\&\mathcal{G}_{\rm{Koiter}}^{\rm{lin}}(v) \,{\rm L}_{y_0} \in{\rm \textbf{Sym}}(2) \quad \textrm{and}\quad 
(\mathcal{R}_{\rm{Koiter}}^{\rm{lin}}(v)-2\,\mathcal{G}_{\rm{Koiter}}^{\rm{lin}}(v) \,{\rm L}_{y_0})\,{\rm L}_{y_0}\in{\rm \textbf{Sym}}(2)
\Bigg\},\notag
\end{align}
 depending on the expressions of $\mathcal{G}_{\rm{Koiter}}^{\rm{lin}}$ and $
\mathcal{R}_{\rm{Koiter}}^{\rm{lin}}$ which we consider, i.e.,  \eqref{formK} and \eqref{formR} \textbf{or} \eqref{equ11} and \eqref{equ12}, respectively. Both admissible sets of functions incorporate a weak reformulation of the symmetry constraint in \eqref{contsyml2}, in the sense that all the derivatives are considered now in the sense of  distributions, and the boundary conditions are to be understood in the sense of traces. The admissible set $\mathcal{A}_{\rm lin}^{\infty}$ defined by \eqref{21l} is a closed subset of the admissible set $\mathcal{A}_1$ given by 	\eqref{admsetK1} and usually considered in the linear Koiter theory, while  $\widehat{\mathcal{A}}_{\rm lin}^{\infty}$ defined by \eqref{21l2} is a closed subset of the admissible set $\mathcal{A}_2$ given by 	\eqref{admsetK2}.

\textit{However,  it is not clear a priori if the  sets $\mathcal{A}_{\rm lin}^{\infty}$ and $\widehat{\mathcal{A}}_{\rm lin}^{\infty}$ are non-empty}, since we still do not know if there exists  $v\in \mathcal{A}_1 $ or $v\in \mathcal{A}_2 $  such that $\mathcal{G}_{\rm{Koiter}}^{\rm{lin}} \,{\rm L}_{y_0} \in{\rm \textbf{Sym}}(2)$ and
$(\mathcal{R}_{\rm{Koiter}}^{\rm{lin}}-2\,\mathcal{G}_{\rm{Koiter}}^{\rm{lin}} \,{\rm L}_{y_0})\,{\rm L}_{y_0}\in{\rm \textbf{Sym}}(2)$. The situation is similar to pure bending type models applied on curved surfaces. In general the set of admissible deformations leaving the first fundamental form invariant, may be empty.

For the moment, let us consider the admissible set  $\mathcal{A}_{\rm lin}^{\infty}$. Similarly as in the linear unconstrained Cosserat shell theory, we rewrite the minimization problem in a weak form. To this aim we define the following operators (we use the same notations, but we point out the dependence on $v$ of all involved quantities)
\begin{align}\label{operatorscon}
\mathcal{E}^{\rm{lin}}_\infty&:\mathcal{A}^{\rm lin}_\infty \to \mathbb{R}^{3\times 3},  \qquad \mathcal{E}^{\rm{lin}}_\infty(v)= [\nabla\Theta \,]^{-T}
[\mathcal{G}_{\rm{Koiter}}^{\rm{lin}}(v)]^\flat [\nabla\Theta \,]^{-1}\,= \  [\nabla\Theta \,]^{-T}
(\mbox{sym}\,\big[ (\nabla y_0)^{T}(\nabla v)\big])^\flat [\nabla\Theta \,]^{-1},\notag\\
A_\infty&:\mathcal{A}^{\rm lin}_\infty \to \mathbb{R}^{3\times 3},\qquad  A_\infty(v)=-\skw( (\nabla v\,|\dd\sum_{\alpha=1,2}\bigl\langle n_0, \partial_{x_\alpha}v\bigr\rangle\, a^\alpha)[\nabla \Theta]^{-1}),\\
\mathcal{K}^{\rm{lin}}_\infty&:\mathcal{A}^{\rm lin}_\infty\to \mathbb{R}^{3\times 3},  \quad\ \ \, \mathcal{K}^{\rm{lin}}_\infty (v)=  \Big(\mbox{axl}(\partial_{x_1}A_\infty(v))\,|\, \mbox{axl}(\partial_{x_2}A_\infty(v))\,|\, 0 \Big) \; [\nabla\Theta \,]^{-1},\notag
\end{align}
the bilinear form
\begingroup
\allowdisplaybreaks
\begin{align}
\mathcal{B}^{\rm{lin}}_\infty&:\mathcal{A}^{\rm lin}_\infty\times \mathcal{A}^{\rm lin}_\infty\to \mathbb{R},\notag\\
\mathcal{B}^{\rm{lin}}_\infty(v,\widetilde{v}):=\int_{\omega}   \!\!\Big[  &\Big(h+{\rm K}\,\dfrac{h^3}{12}\Big)\,
\mathcal{W}_{\mathrm{shell}}^\infty\big(    \mathcal{E}^{\rm{lin}}_\infty(v), \mathcal{E}^{\rm{lin}}_\infty(\widetilde{v})  \big)\notag\\&+  \Big(\dfrac{h^3}{12}\,-{\rm K}\,\dfrac{h^5}{80}\Big)\,
\mathcal{W}_{\mathrm{shell}}^\infty  \big(   \mathcal{E}^{\rm{lin}}_\infty(v)  \, {\rm B}_{y_0} +   {\rm C}_{y_0} \mathcal{K}^{\rm{lin}}_\infty(v), \mathcal{E}^{\rm{lin}}_\infty(\widetilde{v})  \, {\rm B}_{y_0} +   {\rm C}_{y_0} \mathcal{K}^{\rm{lin}}_\infty(\widetilde{v})  \big)  \notag\\&
-\dfrac{h^3}{6} \mathrm{ H}\,\mathcal{W}_{\mathrm{shell}}^\infty  \big(  \mathcal{E}^{\rm{lin}}_\infty(v)  ,
\mathcal{E}^{\rm{lin}}_\infty(\widetilde{v}) {\rm B}_{y_0}+{\rm C}_{y_0}\, \mathcal{K}^{\rm{lin}}_\infty(\widetilde{v})  \big)\\
&-\dfrac{h^3}{6} \mathrm{ H}\,\mathcal{W}_{\mathrm{shell}}^\infty  \big(  \mathcal{E}^{\rm{lin}}_\infty(\widetilde{v})  ,
\mathcal{E}^{\rm{lin}}_\infty(v) {\rm B}_{y_0}+{\rm C}_{y_0}\, \mathcal{K}^{\rm{lin}}_\infty(v) \big)\notag\\&+
\dfrac{h^3}{12}\, \mathcal{W}_{\mathrm{shell}}^\infty  \big(  \mathcal{E}^{\rm{lin}}_\infty(v)  ,
( \mathcal{E}^{\rm{lin}}_\infty(\widetilde{v}) {\rm B}_{y_0}+{\rm C}_{y_0}\, \mathcal{K}^{\rm{lin}}_\infty(\widetilde{v}) ){\rm B}_{y_0} \big)\notag
\\&+
\dfrac{h^3}{12}\, \mathcal{W}_{\mathrm{shell}}^\infty  \big(  \mathcal{E}^{\rm{lin}}_\infty(\widetilde{v})  ,
( \mathcal{E}^{\rm{lin}}_\infty(v) {\rm B}_{y_0}+{\rm C}_{y_0}\, \mathcal{K}^{\rm{lin}}_\infty(v) ){\rm B}_{y_0} \big)\notag\\&+ \,\dfrac{h^5}{80}\,\,
\mathcal{W}_{\mathrm{mp}}^\infty \big((  \mathcal{E}^{\rm{lin}}_\infty(v)  \, {\rm B}_{y_0} +  {\rm C}_{y_0} \mathcal{K}^{\rm{lin}}_\infty(v)  )   {\rm B}_{y_0} ,(  \mathcal{E}^{\rm{lin}}_\infty(\widetilde{v})  \, {\rm B}_{y_0} +  {\rm C}_{y_0} \mathcal{K}^{\rm{lin}}_\infty(\widetilde{v})  )   {\rm B}_{y_0}\big)\notag \\&+ \,\Big(h-{\rm K}\,\dfrac{h^3}{12}\Big)\,
\mathcal{W}_{\mathrm{curv}}\big(  \mathcal{K}^{\rm{lin}}_\infty(v), \mathcal{K}^{\rm{lin}}_\infty(\widetilde{v})  \big)  \notag \\& +  \Big(\dfrac{h^3}{12}\,-{\rm K}\,\dfrac{h^5}{80}\Big)\,
\mathcal{W}_{\mathrm{curv}}\big(  \mathcal{K}^{\rm{lin}}_\infty(v)   {\rm B}_{y_0},\mathcal{K}^{\rm{lin}}_\infty(\widetilde{v})    {\rm B}_{y_0} \,  \big) \notag \\&+ \,\dfrac{h^5}{80}\,\,
\mathcal{W}_{\mathrm{curv}}\big(  \mathcal{K}^{\rm{lin}}_\infty(v)    {\rm B}_{y_0}^2,\mathcal{K}^{\rm{lin}}_\infty(\widetilde{v})      {\rm B}_{y_0}^2   \big)
\Big] {\rm det}(\nabla y_0|n_0)  \,  {\rm da}\notag,
\end{align}
\endgroup
and the linear operator
\begin{align}
{\Pi}^{\rm lin}_\infty:\mathcal{A}^{\rm lin}_\infty\to\mathbb{R}, \qquad\qquad  {\Pi}^{\rm lin}_\infty(\widetilde{v})=\overline{\Pi}_\infty(\widetilde{v}).
\end{align}
Therefore, the weak form of the equilibrium problem of the linear theory of constrained Cosserat shells including terms up to order $O(h^5)$ is to find
$v\in\mathcal{A}^{\rm lin}_\infty$ satisfying 
\begin{align}\label{wfproblemcon}
\mathcal{B}^{\rm{lin}}_\infty(v,\widetilde{v})={\Pi}^{\rm lin}_\infty(\widetilde{v}) \qquad \forall\, \widetilde{v}\in\mathcal{A}^{\rm lin}_\infty.
\end{align}

We endow  $\mathcal{A}^{\rm lin}_\infty$  with the norm
\begin{align}
v=(v_1,v_2,v_3)\mapsto \Big(&\lVert v_1\rVert_{{\rm H}^1(\omega;\mathbb{R})}^2+\lVert v_2\rVert_{{\rm H}^1(\omega;\mathbb{R})}^2+\lVert v_3\rVert_{{\rm H}^2(\omega;\mathbb{R})}^2\notag\\&+\lVert \nabla [\skw( (\nabla v\,|\dd\sum_{\alpha=1,2}\bigl\langle n_0, \partial_{x_\alpha}v\bigr\rangle\, a^\alpha)[\nabla \Theta]^{-1})]]\rVert^2_{{\rm L}^2(\omega)}\Big)^{1/2}.
\end{align} 

Let us recall that one decisive point in the proof of the existence result is the coercivity of the internal energy density, see \cite{GhibaNeffPartIII}:
\begin{lemma}\label{propcoerh5} {\rm [Coercivity in the theory including terms up to order $O(h^5)$]} For sufficiently small values of the thickness $h$ such that  
	\begin{align}\label{rcondh5}
	h\max\{\sup_{x\in\omega}|\kappa_1|, \sup_{x\in\omega}|\kappa_2|\}<\alpha \qquad \text{with}\qquad  \alpha<\sqrt{\frac{2}{3}(29-\sqrt{761})}\simeq 0.97083
	\end{align} 
	and for constitutive coefficients  satisfying  $\mu>0$, $2\,\lambda+\mu> 0$, $b_1>0$, $b_2>0$ and $b_3>0$,   the  energy density
	\begin{align}W(\mathcal{E}_{m,s}, \mathcal{K}_{e,s})=W_{\mathrm{memb}}\big(  \mathcal{E}_{m,s} \big)+W_{\mathrm{memb,bend}}\big(  \mathcal{E}_{m,s} ,\,  \mathcal{K}_{e,s} \big)+W_{\mathrm{bend,curv}}\big(  \mathcal{K}_{e,s}    \big)
	\end{align}
	is coercive in the sense that  there exists a constant   $a_1^+>0$  such that
	$	W(\mathcal{E}_{m,s}, \mathcal{K}_{e,s})\,\geq\, a_1^+\, \big( \lVert \mathcal{E}_{m,s}\rVert ^2 + \lVert \mathcal{K}_{e,s}\rVert ^2\,\big)$,
	where
	$a_1^+$ depends on the constitutive coefficients. In fact,  the following inequality holds true
	\begin{align}
	W(\mathcal{E}_{m,s}, \mathcal{K}_{e,s})\,\geq\, a_1^+\, \big( \lVert \mathcal{E}_{m,s}\rVert ^2 +\lVert
	\mathcal{E}_{m,s}{\rm B}_{y_0}+{\rm C}_{y_0}\, \mathcal{K}_{e,s}  \rVert ^2+ \lVert \mathcal{K}_{e,s}\rVert ^2\,\big).
	\end{align}
\end{lemma}

For the proof of the existence of the solution we need the next estimate given by the following result.
\begin{lemma}\label{lemmaLy}
Let there be given a domain $\omega\subset \mathbb{R}^2$ and assume  that the initial configuration of the curved shell is defined by  a continuous injective mapping  $y_0\in {\rm C}^{3}(\overline{\omega}, \mathbb{R}^3)$  such that the two vectors $a_\alpha=\partial_{x_\alpha}y_0$, $\alpha=1,2$ are linear independent at all points of $\overline{\omega}$. There exists a constant $c>0$ such that
\begin{align}\label{u10}	\lVert
\mathcal{G}_{\rm{Koiter}}^{\rm{lin}}(v) \rVert^2&+\lVert \mathcal{R}_{\rm{Koiter}}^{\rm{lin}}({v})-2\,\mathcal{G}_{\rm{Koiter}}^{\rm{lin}}({v}) \,{\rm L}_{y_0} \rVert^2\geq c\left(\lVert\mathcal{G}_{\rm{Koiter}}^{\rm{lin}}(v) \rVert^2+\lVert \mathcal{R}_{\rm{Koiter}}^{\rm{lin}}({v})\rVert^2\right),
\end{align}
everywhere in $\omega$.
\end{lemma}
\begin{proof}
	We split our discussion. On one hand, in all the points of $\omega$ for which ${\rm L}_{y_0}=0$ it follows that 
	\begin{align}\label{u8}	\lVert
	\mathcal{G}_{\rm{Koiter}}^{\rm{lin}}(v) \rVert^2&+\lVert \mathcal{R}_{\rm{Koiter}}^{\rm{lin}}({v})-2\,\mathcal{G}_{\rm{Koiter}}^{\rm{lin}}({v}) \,{\rm L}_{y_0} \rVert^2=\lVert\mathcal{G}_{\rm{Koiter}}^{\rm{lin}}(v) \rVert^2+\lVert \mathcal{R}_{\rm{Koiter}}^{\rm{lin}}({v})\rVert^2.
	\end{align}
	
On the other hand, 	due to the properties of the Frobenius norm (including sub-multiplicity), in all the points of $\omega$ for which ${\rm L}_{y_0}\neq 0$ we deduce
	\begin{align}\label{u9}	\lVert
	\mathcal{G}_{\rm{Koiter}}^{\rm{lin}}(v) \rVert^2&+\lVert \mathcal{R}_{\rm{Koiter}}^{\rm{lin}}({v})-2\,\mathcal{G}_{\rm{Koiter}}^{\rm{lin}}({v}) \,{\rm L}_{y_0} \rVert^2\notag\\&=\frac{1}{2}\lVert
	\mathcal{G}_{\rm{Koiter}}^{\rm{lin}}(v) \rVert^2+\frac{1}{2}\lVert
	\mathcal{G}_{\rm{Koiter}}^{\rm{lin}}(v) \rVert^2+\lVert \mathcal{R}_{\rm{Koiter}}^{\rm{lin}}({v})-2\,\mathcal{G}_{\rm{Koiter}}^{\rm{lin}}({v}) \,{\rm L}_{y_0} \rVert^2\notag
	\\
		&= \frac{1}{2}\lVert
	\mathcal{G}_{\rm{Koiter}}^{\rm{lin}}(v) \rVert^2+\frac{1}{2\,\lVert 2\,{\rm L}_{y_0} \rVert^2}\lVert
\mathcal{G}_{\rm{Koiter}}^{\rm{lin}}(v) \rVert^2 \lVert 2\,{\rm L}_{y_0} \rVert^2+\lVert \mathcal{R}_{\rm{Koiter}}^{\rm{lin}}({v})-2\,\mathcal{G}_{\rm{Koiter}}^{\rm{lin}}({v}) \,{\rm L}_{y_0} \rVert^2\notag
	\\
	&\geq \frac{1}{2}\lVert
	\mathcal{G}_{\rm{Koiter}}^{\rm{lin}}(v) \rVert^2+\frac{1}{2\,\lVert 2\,{\rm L}_{y_0} \rVert^2}\lVert
	2\,\mathcal{G}_{\rm{Koiter}}^{\rm{lin}}(v)\,{\rm L}_{y_0} \rVert^2+\lVert \mathcal{R}_{\rm{Koiter}}^{\rm{lin}}({v})-2\,\mathcal{G}_{\rm{Koiter}}^{\rm{lin}}({v}) \,{\rm L}_{y_0} \rVert^2
	\\
	&\geq \min\{\frac{1}{2},\frac{1}{2\,\lVert 2\,{\rm L}_{y_0} \rVert^2},1\}\left(\lVert
	\mathcal{G}_{\rm{Koiter}}^{\rm{lin}}(v) \rVert^2+\lVert
	2\,\mathcal{G}_{\rm{Koiter}}^{\rm{lin}}(v)\,{\rm L}_{y_0} \rVert^2+\lVert \mathcal{R}_{\rm{Koiter}}^{\rm{lin}}({v})-2\,\mathcal{G}_{\rm{Koiter}}^{\rm{lin}}({v}) \,{\rm L}_{y_0} \rVert^2\right)\notag\\
&= \frac{1}{2}\min\{1,\frac{1}{\lVert 2\,{\rm L}_{y_0} \rVert^2}\}\left(\lVert
\mathcal{G}_{\rm{Koiter}}^{\rm{lin}}(v) \rVert^2+\lVert
2\,\mathcal{G}_{\rm{Koiter}}^{\rm{lin}}(v)\,{\rm L}_{y_0} \rVert^2+\lVert \mathcal{R}_{\rm{Koiter}}^{\rm{lin}}({v})-2\,\mathcal{G}_{\rm{Koiter}}^{\rm{lin}}({v}) \,{\rm L}_{y_0} \rVert^2\right)\notag.\end{align}
Of course $\lVert 2\,{\rm L}_{y_0} \rVert^2$ depends on the points in $\overline{\omega}$. Since
$y_0\in {\rm C}^{3}(\overline{\omega}, \mathbb{R}^3)$ is such that the two vectors $a_\alpha=\partial_{x_\alpha}y_0$, $\alpha=1,2$ are linear independent at all points of $\overline{\omega}$ it follows that $
\det[\nabla_x\Theta(0)] \geq\, c_0 >0
$ at all points of $\overline{\omega}$. Moreover, we have
\begin{align}
{\lVert {\rm L}_{y_0} \rVert}=\lVert {\rm I}_{y_0}^{-1} \, {\rm II}_{y_0} \rVert\leq \lVert {\rm I}_{y_0}^{-1} \rVert\, \lVert {\rm II}_{y_0} \rVert= \frac{1}{\det  {\rm I}_{y_0}}\lVert {\rm Cof}\,{\rm I}_{y_0}\rVert\, \lVert {\rm II}_{y_0} \rVert
\end{align}
and since $y_0\in {\rm C}^2(\overline{\omega};\mathbb{R}^3)$ implies ${\rm Cof}\,{\rm I}_{y_0}\in {\rm C}^1(\overline{\omega};\mathbb{R}^{2\times 2})$ and ${\rm II}_{y_0} \in {\rm C}(\overline{\omega};\mathbb{R}^{2\times 2})$, together with $\det[\nabla_x\Theta(0)] \geq\, c_0 >0$ it follows that there is $c>0$ independent of $x\in\overline{\omega}$ such that
\begin{align}
\lVert {\rm L}_{y_0} \rVert\leq  \frac{1}{\det  {\rm I}_{y_0}}\lVert {\rm Cof}\,{\rm I}_{y_0}\rVert\, \lVert {\rm II}_{y_0} \rVert<c,
\end{align}
i.e., there exists another constant $c>0$ such that
\begin{align}\label{uLy}
\frac{1}{\lVert 2\,{\rm L}_{y_0} \rVert^2}>c.
\end{align}

	Therefore, using \eqref{u9} and \eqref{uLy} we deduce that there exists a constant $c>0$ such that
		\begin{align}\label{u10l}	\lVert
	\mathcal{G}_{\rm{Koiter}}^{\rm{lin}}(v) \rVert^2&+\lVert \mathcal{R}_{\rm{Koiter}}^{\rm{lin}}({v})-2\,\mathcal{G}_{\rm{Koiter}}^{\rm{lin}}({v}) \,{\rm L}_{y_0} \rVert^2\\&\geq \frac{1}{2}\min\{1,\frac{1}{\lVert 2\,{\rm L}_{y_0} \rVert^2}\}\left(\lVert
	\mathcal{G}_{\rm{Koiter}}^{\rm{lin}}(v) \rVert^2+\lVert
	2\,\mathcal{G}_{\rm{Koiter}}^{\rm{lin}}(v)\,{\rm L}_{y_0} \rVert^2+\lVert \mathcal{R}_{\rm{Koiter}}^{\rm{lin}}({v})-2\,\mathcal{G}_{\rm{Koiter}}^{\rm{lin}}({v}) \,{\rm L}_{y_0} \rVert^2\right)\notag\\&\geq c\,\left(\lVert
	\mathcal{G}_{\rm{Koiter}}^{\rm{lin}}(v) \rVert^2+\lVert
	2\,\mathcal{G}_{\rm{Koiter}}^{\rm{lin}}(v)\,{\rm L}_{y_0} \rVert^2+\lVert \mathcal{R}_{\rm{Koiter}}^{\rm{lin}}({v})-2\,\mathcal{G}_{\rm{Koiter}}^{\rm{lin}}({v}) \,{\rm L}_{y_0} \rVert^2\right).\notag
	\end{align}
Now we apply the reverse triangle inequality to see that 
\begin{align}
\lVert \mathcal{R}_{\rm{Koiter}}^{\rm{lin}}({v})-2\,\mathcal{G}_{\rm{Koiter}}^{\rm{lin}}({v}) \,{\rm L}_{y_0} \rVert\geq\lVert \mathcal{R}_{\rm{Koiter}}^{\rm{lin}}({v})\rVert-	\lVert
2\,\mathcal{G}_{\rm{Koiter}}^{\rm{lin}}(v)\,{\rm L}_{y_0} \rVert,
\end{align}
hence
\begin{align}
	\lVert
2\,\mathcal{G}_{\rm{Koiter}}^{\rm{lin}}(v)\,{\rm L}_{y_0} \rVert+\lVert \mathcal{R}_{\rm{Koiter}}^{\rm{lin}}({v})-2\,\mathcal{G}_{\rm{Koiter}}^{\rm{lin}}({v}) \,{\rm L}_{y_0} \rVert\geq\lVert \mathcal{R}_{\rm{Koiter}}^{\rm{lin}}({v})\rVert.
\end{align}

Inserting this in \eqref{u10l} 
it follows that there exists a positive constant $c>0$ such that
		\begin{align}\label{u10l1}	\lVert
	\mathcal{G}_{\rm{Koiter}}^{\rm{lin}}(v) \rVert^2&+\lVert \mathcal{R}_{\rm{Koiter}}^{\rm{lin}}({v})-2\,\mathcal{G}_{\rm{Koiter}}^{\rm{lin}}({v}) \,{\rm L}_{y_0} \rVert^2 \geq c\left(\lVert\mathcal{G}_{\rm{Koiter}}^{\rm{lin}}(v) \rVert^2+\lVert \mathcal{R}_{\rm{Koiter}}^{\rm{lin}}({v})\rVert^2\right),
	\end{align}
	everywhere in $\omega$, and the proof is complete. 
\end{proof}

After having Lemma \ref{lemmaLy}, other similar arguments as in the unconstrained linear Cosserat shell model allow us to formulate the following existence result.
\begin{theorem}\label{th1cond}{\rm [A conditional existence result for the theory including terms up to order $O(h^5)$  on a general surface]}
Let there be given a domain $\omega\subset \mathbb{R}^2$ and an injective mapping $y_0\in {\rm C}^{3}(\overline{\omega}, \mathbb{R}^3)$  such that the two vectors $a_\alpha=\partial_{x_\alpha}y_0$, $\alpha=1,2$, are linear independent at all points of $\overline{\omega}$. Assume that 	 the admissible set $\mathcal{A}^{\rm lin}_\infty$ is non-empty and that the linear operator
	$
	{\Pi}^{\rm lin}_\infty$ is bounded.	
	Then, for sufficiently small values of the thickness $h$  {such that condition \eqref{rcondh5} is satisfied}
	and for constitutive coefficients  such that $\mu>0$, $2\,\lambda+\mu> 0$, $b_1>0$, $b_2>0$ and $b_3>0$, the  problem \eqref{wfproblemcon} admits a unique solution
	$v\in  \mathcal{A}^{\rm lin}_\infty$.
\end{theorem}
\begin{proof} 
	First, we show that for sufficiently small values of the thickness $h$ such that   \eqref{rcondh5} is satisfied 
	and for constitutive coefficients  satisfying  $\mu>0$, $2\,\lambda+\mu> 0$, $b_1>0$, $b_2>0$ and $b_3>0$,   the  bilinear form $\mathcal{B}^{\rm{lin}}_\infty$
	is coercive. According to Proposition 3.3. from \cite{GhibaNeffPartIII}, see also Lemma \ref{propcoerh5}, and considering that energy terms containing $\mu_c$ are absent everywhere,   there exists a constant   $a_1^+>0$  such that
	\begin{equation}\label{26bis}
	W^{\infty}(\mathcal{E}_{ \infty }^{\rm lin}(v), \mathcal{K}_{ \infty }^{\rm lin}(v))\,\geq\, a_1^+\, \big( \lVert \mathcal{E}_{ \infty }^{\rm lin}(v)\rVert ^2 + \lVert
	\mathcal{E}_{ \infty }^{\rm lin}(v){\rm B}_{y_0}+{\rm C}_{y_0}\, \mathcal{K}_{ \infty }^{\rm lin} (v) \rVert ^2+\lVert \mathcal{K}_{ \infty }^{\rm lin}(v)\rVert ^2\,\big)  \quad \forall\, v\in \mathcal{A}_{ \infty }^{\rm lin},
	\end{equation}
	where
	$a_1^+$ depends on the constitutive coefficients (but not on the Cosserat couple modulus $\mu_{\rm c}$). 
Therefore, it follows that   there exists a constant   $a_1^+>0$  such that
\begin{align}\label{u1}	\mathcal{B}_{ \infty }^{\rm lin}(v,{v}) =\int_\omega 	W^{\infty}(\mathcal{E}_{ \infty }^{\rm lin}(v), \mathcal{K}_{ \infty }^{\rm lin}(v)){\rm det}(\nabla y_0|n_0)    {\rm da}\geq\, a_1^+\, \big( \lVert \mathcal{E}_{ \infty }^{\rm lin}(v)\rVert ^2_{{\rm L}^2(\omega)} + \lVert \mathcal{K}_{ \infty }^{\rm lin}(v)\rVert ^2_{{\rm L}^2(\omega)}\,\big) \quad \forall\, v\in \mathcal{A}_{ \infty }^{\rm lin}.
\end{align}
Since 
\begin{align}\label{u2}	
\mathcal{E}_{ \infty }^{\rm lin}(v)=[\nabla\Theta \,]^{-T}
[\mathcal{G}_{\rm{Koiter}}^{\rm{lin}}(v)]^\flat [\nabla\Theta \,]^{-1},
\end{align}
we obtain
\begin{align}\label{u3}	
\lVert \mathcal{E}_{ \infty }^{\rm lin}(v)\rVert^2&= \bigl\langle [\nabla\Theta \,]^{-T}
[\mathcal{G}_{\rm{Koiter}}^{\rm{lin}}(v)]^\flat [\nabla\Theta \,]^{-1},[\nabla\Theta \,]^{-T}
[\mathcal{G}_{\rm{Koiter}}^{\rm{lin}}(v)]^\flat [\nabla\Theta \,]^{-1}\bigr\rangle\notag\\
&\geq  \bigl\langle [\nabla\Theta \,]^{-1}[\nabla\Theta \,]^{-T}
[\mathcal{G}_{\rm{Koiter}}^{\rm{lin}}(v)]^\flat [\nabla\Theta \,]^{-1},
[\mathcal{G}_{\rm{Koiter}}^{\rm{lin}}(v)]^\flat [\nabla\Theta \,]^{-1}\bigr\rangle
\notag\\
&\geq \lambda_{\rm min}(\widehat{\rm I}^{-1}_{y_0}) \bigl\langle 
[\mathcal{G}_{\rm{Koiter}}^{\rm{lin}}(v)]^\flat ,
[\mathcal{G}_{\rm{Koiter}}^{\rm{lin}}(v)]^\flat [\nabla\Theta \,]^{-1}[\nabla\Theta \,]^{-T}\bigr\rangle
\\
&\geq \lambda_{\rm min}(\widehat{\rm I}^{-1}_{y_0})^2 \lVert
\mathcal{G}_{\rm{Koiter}}^{\rm{lin}}(v) \rVert^2,\notag
\end{align}
 where $0<\lambda_{\rm min}(\widehat{\rm I}^{-1}_{y_0})=\min\{\lambda_{\rm min}({\rm I}^{-1}_{y_0}),1\}\leq 1$ is the smallest eigenvalue of the positive definite matrix $\widehat{\rm I}^{-1}_{y_0}$. 

Similarly, from
\begin{align}\label{u4}	
\mathcal{E}_{ \infty }^{\rm lin}(v){\rm B}_{y_0}+{\rm C}_{y_0}\, \mathcal{K}_{ \infty }^{\rm lin} (v) =[\nabla\Theta \,]^{-T} [\mathcal{R}_{\rm{Koiter}}^{\rm{lin}}({v})-2\,\mathcal{G}_{\rm{Koiter}}^{\rm{lin}}({v}) \,{\rm L}_{y_0} ]^\flat
[\nabla\Theta \,]^{-1},
\end{align}
we obtain
\begin{align}\label{u5}	
\lVert \mathcal{E}_{ \infty }^{\rm lin}(v){\rm B}_{y_0}+{\rm C}_{y_0}\, \mathcal{K}_{ \infty }^{\rm lin} (v) \rVert^2\geq \lambda_{\rm min}^2(\widehat{\rm I}^{-1}_{y_0})\lVert \mathcal{R}_{\rm{Koiter}}^{\rm{lin}}({v})-2\,\mathcal{G}_{\rm{Koiter}}^{\rm{lin}}({v}) \,{\rm L}_{y_0} \rVert^2,
\end{align}
while 
\begin{align}\label{u6}	
\lVert \mathcal{K}_{ \infty }^{\rm lin}(v)\rVert ^2&\geq \lambda_{\rm min}(\widehat{\rm I}^{-1}_{y_0})(\lVert\mbox{axl}(\partial_{x_1}A_\infty(v))\rVert^2+\lVert\mbox{axl}(\partial_{x_2}A_\infty(v))\rVert^2)= \frac{\lambda_{\rm min}(\widehat{\rm I}^{-1}_{y_0})}{2}(\lVert\partial_{x_1}A_\infty(v)\rVert^2+\lVert\partial_{x_2}A_\infty(v)\rVert^2)\notag\\&= \frac{\lambda_{\rm min}(\widehat{\rm I}^{-1}_{y_0})}{2}\lVert\nabla\,A_\infty(v)\rVert^2.
\end{align}

Let us prove that under the hypothesis of the theorem $\lambda_{\rm min}(\widehat{\rm I}^{-1}_{y_0})$ is bounded below  in $\overline{\omega}$ by a positive constant. We notice again that since
$y_0\in {\rm C}^{3}(\omega, \mathbb{R}^3)$ is such that the two vectors $a_\alpha=\partial_{x_\alpha}y_0$, $\alpha=1,2$ are linear independent at all points of $\overline{\omega}$ it follows that $
\det[\nabla_x\Theta(0)] \geq\, c_0 >0
$ at all points of $\overline{\omega}$.
Moreover, we have \begin{align}
2\,\lambda_{\rm max}^2({\rm I}^{-1}_{y_0})\geq \lambda_{\rm min}^2({\rm I}^{-1}_{y_0})+\lambda_{\rm max}^2({\rm I}^{-1}_{y_0})=\lVert {\rm I}^{-1}_{y_0}\rVert^2,
\end{align}
where $\lambda_{\rm max}^2({\rm I}^{-1}_{y_0})$ is the largest eigenvalue of the positive definite matrix $\widehat{\rm I}^{-1}_{y_0}$, and
\begin{align}
\lambda_{\rm min}({\rm I}^{-1}_{y_0})=\frac{\det {\rm I}^{-1}_{y_0}}{\lambda_{\rm max}({\rm I}^{-1}_{y_0})}=\frac{1}{\lambda_{\rm max}({\rm I}^{-1}_{y_0})\det {\rm I}_{y_0}}\geq \frac{1}{\sqrt{2}\,\lVert{\rm I}^{-1}_{y_0}\rVert \det {\rm I}_{y_0}}=\frac{1}{\sqrt{2}\,\lVert {\rm Cof}\,{\rm I}_{y_0}\rVert}.
\end{align}
Since $y_0\in C^2(\overline{\omega};\mathbb{R}^3)$ implies ${\rm Cof}\,{\rm I}_{y_0}\in C^1(\overline{\omega};\mathbb{R}^{2\times 2})$, we have that ${\rm Cof}\,{\rm I}_{y_0}$ is bounded above. Hence, there exists a positive constant $c>0$ such that 
\begin{align}
\lambda_{\rm min}({\rm I}^{-1}_{y_0})=\frac{\det {\rm I}^{-1}_{y_0}}{\lambda_{\rm max}({\rm I}^{-1}_{y_0})}\geq\frac{1}{\sqrt{2}\,\lVert {\rm Cof}\,{\rm I}_{y_0}\rVert}\geq c>0,
\end{align}
i.e., that $\lambda_{\rm min}(\widehat{\rm I}^{-1}_{y_0})>0$ is bounded below over $\overline{\omega}$ by a positive constant.

Thus, from  \eqref{u3}, \eqref{u5} and \eqref{u6}, we deduce that there exists a positive constant $c>0$ such that
\begin{align}\label{u7}	\mathcal{B}_{ \infty }^{\rm lin}(v,{v})\geq c\,\big(\lVert
\mathcal{G}_{\rm{Koiter}}^{\rm{lin}}(v) \rVert^2_{{\rm L}^2(\omega)}+\lVert \mathcal{R}_{\rm{Koiter}}^{\rm{lin}}({v})-2\,\mathcal{G}_{\rm{Koiter}}^{\rm{lin}}({v}) \,{\rm L}_{y_0} \rVert^2_{{\rm L}^2(\omega)}+\lVert\nabla A_\infty(v)\rVert^2_{{\rm L}^2(\omega)}\big)
\quad \forall\, v\in \mathcal{A}_{ \infty }^{\rm lin}.\end{align}

 The estimate \eqref{u10} together with \eqref{u1} and \eqref{u6} lead us to the existence of a positive constant $c>0$ such that
 \begin{align}\label{u11}	\mathcal{B}_{ \infty }^{\rm lin}(v,{v})\geq\, c\, \big( \lVert\mathcal{G}_{\rm{Koiter}}^{\rm{lin}}(v) \rVert^2_{{\rm L}^2(\omega)}+\lVert \mathcal{R}_{\rm{Koiter}}^{\rm{lin}}({v})\rVert^2_{{\rm L}^2(\omega)}+\lVert\nabla A_\infty(v)\rVert^2_{{\rm L}^2(\omega)}\,\big) \quad \forall\, v\in \mathcal{A}_{ \infty }^{\rm lin}.
 \end{align}

 Using the Korn inequality given by Theorem \ref{Kornlr}, we have that  there exists a constant $c>0$ such that
 \begin{align}
 \lVert \mathcal{G}_{\rm{Koiter}}^{\rm{lin}} \rVert_{{\rm L}^2(\omega)}^2+\lVert \mathcal{R}_{\rm{Koiter}}^{\rm{lin}} \rVert_{{\rm L}^2(\omega)}^2\geq  c\,(\lVert v_1\rVert_{{\rm H}^1(\omega;\mathbb{R})}^2+\lVert v_2\rVert_{{\rm H}^1(\omega;\mathbb{R})}^2+\lVert v_3\rVert_{{\rm H}^2(\omega;\mathbb{R})}^2)\qquad \forall \ v\in \mathcal{A}_{ \infty }^{\rm lin}\subset \mathcal{A}_1.
 \end{align}
As a conclusion, we have that there exists a constant $c>0$ such that
 \begin{align}\label{u12}	\mathcal{B}_{ \infty }^{\rm lin}(v,{v})\geq\, c\, \big(\lVert v_1\rVert_{{\rm H}^1(\omega;\mathbb{R})}^2+\lVert v_2\rVert_{{\rm H}^1(\omega;\mathbb{R})}^2+\lVert v_3\rVert_{{\rm H}^2(\omega;\mathbb{R})}^2+\lVert\nabla A_\infty(v)\rVert^2_{{\rm L}^2(\omega)}\,\big) \quad \forall\, v\in \mathcal{A}_{ \infty }^{\rm lin}.
 \end{align}
 In consequence, the bilinear form is coercive on $\mathcal{A}_{ \infty }^{\rm lin}$ and the Lax-Milgram theorem leads us to the conclusion of the theorem.
\end{proof}
\begin{theorem}\label{th1cond2}{\rm [A conditional existence result for the theory including terms up to order $O(h^5)$  for shells whose middle surface has little regularity]}
	Let there be given a domain $\omega\subset \mathbb{R}^2$ and an injective mapping $y_0\in {\rm H}^{2,\infty}(\omega, \mathbb{R}^3)$ such that the two vectors $a_\alpha=\partial_{x_\alpha}y_0$, $\alpha=1,2$ are linear independent at all points of $\overline{\omega}$. 	Assume that  the admissible set $\widehat{\mathcal{A}}^{\rm lin}_\infty$ is non-empty and that the linear operator
	$
	{\Pi}^{\rm lin}_\infty$ is bounded.	
	Then, for sufficiently small values of the thickness $h$  {such that condition \eqref{rcondh5} is satisfied}
	and for constitutive coefficients  such that $\mu>0$, $2\,\lambda+\mu> 0$, $b_1>0$, $b_2>0$ and $b_3>0$, the  problem \eqref{wfproblemcon} admits a unique solution
	$v\in  \widehat{\mathcal{A}}^{\rm lin}_\infty$.
\end{theorem}
\begin{proof} The proof is similar with that of the first existence result Theorem \ref{th1cond}. There are only three important differences. Before proceeding to discuss these differences, let us notice that since
	$y_0\in {\rm H}^{2,\infty}(\omega, \mathbb{R}^3)$ is such that the two vectors $a_\alpha=\partial_{x_\alpha}y_0$, $\alpha=1,2$ are linear independent at all points of $\overline{\omega}$ it follows that $
	\det[\nabla_x\Theta(0)] \geq\, c_0 >0
	$ at all points of $\overline{\omega}$.
	
	The first difference is regarding the proof of the estimates
	\begin{align}\label{u5n}	
	\lVert \mathcal{E}_{ \infty }^{\rm lin}(v){\rm B}_{y_0}+{\rm C}_{y_0}\, \mathcal{K}_{ \infty }^{\rm lin} (v) \rVert^2_{{\rm L}^2(\omega)}&\geq c\,\lVert \mathcal{R}_{\rm{Koiter}}^{\rm{lin}}({v})-2\,\mathcal{G}_{\rm{Koiter}}^{\rm{lin}}({v}) \,{\rm L}_{y_0} \rVert^2_{{\rm L}^2(\omega)},\\	
	\lVert \mathcal{K}_{ \infty }^{\rm lin}(v)\rVert ^2_{{\rm L}^2(\omega)}&\geq c\, \lVert\nabla A_\infty(v)\rVert^2_{{\rm L}^2(\omega)},\notag
	\end{align}
	with $c>0$ a positive constant over $\overline{\omega}$, since in the proof of Lemma \ref{lemmaLy} which was used in the proof of Theorem \ref{th1cond} we have used that
	$y_0\in {\rm C}^{3}(\overline{\omega}, \mathbb{R}^3)$, while in the hypothesis of the theorem under discussion we have $y_0\in {\rm H}^{2,\infty}(\omega, \mathbb{R}^3)$.

	We still know from \eqref{u5} and \eqref{u6} that
	\begin{align}\label{u5nn}	
	\lVert \mathcal{E}_{ \infty }^{\rm lin}(v){\rm B}_{y_0}+{\rm C}_{y_0}\, \mathcal{K}_{ \infty }^{\rm lin} (v) \rVert^2\geq \lambda_{\rm min}(\widehat{\rm I}^{-1}_{y_0})\lVert \mathcal{R}_{\rm{Koiter}}^{\rm{lin}}({v})-2\,\mathcal{G}_{\rm{Koiter}}^{\rm{lin}}({v}) \,{\rm L}_{y_0} \rVert^2,
	\end{align}
and
	\begin{align}\label{u6nn}	
	\lVert \mathcal{K}_{ \infty }^{\rm lin}(v)\rVert ^2&\geq  \frac{\lambda_{\rm min}(\widehat{\rm I}^{-1}_{y_0})}{2}\lVert\nabla A_\infty(v)\rVert^2,
	\end{align}
	since they are independent by the assumption
	$y_0\in {\rm C}^{3}(\overline{\omega}, \mathbb{R}^3)$ and $\lambda_{\rm min}^2(\widehat{\rm I}^{-1}_{y_0})\leq 1$. In addition, for 
	$
	\det[\nabla_x\Theta(0)] \geq\, c_0 >0
	$ it is true that
$
	\lambda_{\rm min}({\rm I}^{-1}_{y_0})\geq \frac{1}{\sqrt{2}\,\lVert {\rm Cof}\,{\rm I}_{y_0}\rVert}.
$
	Therefore, $y_0\in {\rm H}^{2,\infty}(\omega, \mathbb{R}^3)$ implies ${\rm Cof}\,{\rm I}_{y_0} \in {\rm L}^\infty(\omega)$ and 
	\begin{align}	
\int_{\omega}\lVert \mathcal{E}_{ \infty }^{\rm lin}(v){\rm B}_{y_0}+{\rm C}_{y_0}\, \mathcal{K}_{ \infty }^{\rm lin} (v) \rVert^2\, {\rm da}&\geq \int_{\omega} \frac{1}{\sqrt{2}\,\lVert {\rm Cof}\,{\rm I}_{y_0}\rVert}\lVert \mathcal{R}_{\rm{Koiter}}^{\rm{lin}}({v})-2\,\mathcal{G}_{\rm{Koiter}}^{\rm{lin}}({v}) \,{\rm L}_{y_0} \rVert^2\, {\rm da}\\
&\geq \frac{1}{\sqrt{2}\,\lVert {\rm Cof}\,{\rm I}_{y_0}\rVert_{L^\infty(\omega)}} \int_{\omega} \lVert \mathcal{R}_{\rm{Koiter}}^{\rm{lin}}({v})-2\,\mathcal{G}_{\rm{Koiter}}^{\rm{lin}}({v}) \,{\rm L}_{y_0} \rVert^2\, {\rm da}\notag
\end{align}
i.e., there exists a constant $c>0$ such that
\begin{align}	
\lVert \mathcal{E}_{ \infty }^{\rm lin}(v){\rm B}_{y_0}+{\rm C}_{y_0}\, \mathcal{K}_{ \infty }^{\rm lin} (v) \rVert^2_{{\rm L}^2(\omega)}
\geq c\,  \lVert \mathcal{R}_{\rm{Koiter}}^{\rm{lin}}({v})-2\,\mathcal{G}_{\rm{Koiter}}^{\rm{lin}}({v}) \,{\rm L}_{y_0} \rVert^2_{{\rm L}^2(\omega)}.
\end{align}
In a similar way it follows that there exists a constant $c>0$ such that
	\begin{align}
	\lVert \mathcal{K}_{ \infty }^{\rm lin}(v)\rVert ^2_{{\rm L}^2(\omega)} &\geq  c\,\lVert\nabla A_\infty(v)\rVert^2_{{\rm L}^2(\omega)}.
	\end{align}
	
	The second difference is regarding the needed estimate
	\begin{align}
	\lVert
	\mathcal{G}_{\rm{Koiter}}^{\rm{lin}}(v) \rVert^2_{{\rm L}^2(\omega)}+\lVert \mathcal{R}_{\rm{Koiter}}^{\rm{lin}}({v})-2\,\mathcal{G}_{\rm{Koiter}}^{\rm{lin}}({v}) \,{\rm L}_{y_0} \rVert^2_{{\rm L}^2(\omega)}\geq c\, \big(\big(\lVert
	\mathcal{G}_{\rm{Koiter}}^{\rm{lin}}(v) \rVert^2_{{\rm L}^2(\omega)}+\lVert \mathcal{R}_{\rm{Koiter}}^{\rm{lin}}({v}) \rVert^2_{{\rm L}^2(\omega)}\big),
	\end{align}
	where $c$ is a positive constant. 
	However, the estimates \eqref{u8} and \eqref{u9} are true for  $y_0\in {\rm H}^{2,\infty}(\omega, \mathbb{R}^3)$, too.
Moreover, since \begin{align}
	\lVert {\rm L}_{y_0} \rVert\leq  \frac{1}{\det  {\rm I}_{y_0}}\lVert {\rm Cof}\,{\rm I}_{y_0}\rVert\, \lVert {\rm II}_{y_0} \rVert
	\end{align}
	and $
	\det[\nabla_x\Theta(0)] \geq\, c_0 >0
	$ at all points of $\overline{\omega}$, from \eqref{u8} and \eqref{u9}
	we deduce the estimate
		\begin{align}\label{u9n}	\lVert
	\mathcal{G}_{\rm{Koiter}}^{\rm{lin}}(v) \rVert^2&+\lVert \mathcal{R}_{\rm{Koiter}}^{\rm{lin}}({v})-2\,\mathcal{G}_{\rm{Koiter}}^{\rm{lin}}({v}) \,{\rm L}_{y_0} \rVert^2\notag\\
	&\geq \begin{cases}
	\left(\lVert
	\mathcal{G}_{\rm{Koiter}}^{\rm{lin}}(v) \rVert^2+\lVert \mathcal{R}_{\rm{Koiter}}^{\rm{lin}}({v}) \rVert^2\right)&\text{if}\ \ {\rm L}_{y_0}=0 ,\\\dd\frac{1}{2}\min\{\frac{\det  {\rm I}_{y_0}}{4\, \lVert {\rm Cof}\,{\rm I}_{y_0}\rVert^2\, \lVert {\rm II}_{y_0} \rVert^2},1\}\left(\lVert
	\mathcal{G}_{\rm{Koiter}}^{\rm{lin}}(v) \rVert^2+\lVert \mathcal{R}_{\rm{Koiter}}^{\rm{lin}}({v}) \rVert^2\right)&\text{if}\ \ {\rm L}_{y_0}\neq 0.
	\end{cases}\notag\\
	&\geq \begin{cases}
	\left(\lVert
	\mathcal{G}_{\rm{Koiter}}^{\rm{lin}}(v) \rVert^2+\lVert \mathcal{R}_{\rm{Koiter}}^{\rm{lin}}({v}) \rVert^2\right)&\text{if}\ \ {\rm L}_{y_0}=0 ,\\\dd\frac{1}{2}\min\{\frac{c}{4\, \lVert {\rm Cof}\,{\rm I}_{y_0}\rVert^2\, \lVert {\rm II}_{y_0} \rVert^2},1\}\left(\lVert
	\mathcal{G}_{\rm{Koiter}}^{\rm{lin}}(v) \rVert^2+\lVert \mathcal{R}_{\rm{Koiter}}^{\rm{lin}}({v}) \rVert^2\right)&\text{if}\ \ {\rm L}_{y_0}\neq 0,
	\end{cases}
	\notag\\
	&\geq \dd\min\{\frac{c}{8\, \lVert {\rm Cof}\,{\rm I}_{y_0}\rVert^2\, \lVert {\rm II}_{y_0} \rVert^2},\frac{1}{2}\}\left(\lVert
	\mathcal{G}_{\rm{Koiter}}^{\rm{lin}}(v) \rVert^2+\lVert \mathcal{R}_{\rm{Koiter}}^{\rm{lin}}({v}) \rVert^2\right)\end{align}
	where $c>0$ is a positive constant. Hence, since $y_0\in {\rm H}^{2,\infty}(\omega, \mathbb{R}^3)$ implies ${\rm Cof}\,{\rm I}_{y_0} \in {\rm L}^\infty(\omega)$  and ${\rm II}_{y_0} \in {\rm L}^\infty(\omega)$ we deduce
		\begin{align}\label{u9nn}\int_\omega	\Big(\lVert
	\mathcal{G}_{\rm{Koiter}}^{\rm{lin}}(v) \rVert^2&+\lVert \mathcal{R}_{\rm{Koiter}}^{\rm{lin}}({v})-2\,\mathcal{G}_{\rm{Koiter}}^{\rm{lin}}({v}) \,{\rm L}_{y_0} \rVert^2\Big) {\rm da}\notag\\
	&\geq \min\{\frac{c}{8\, \lVert {\rm Cof}\,{\rm I}_{y_0}\rVert^2_{ {\rm L}^\infty(\omega)}\, \lVert {\rm II}_{y_0} \rVert^2_{ {\rm L}^\infty(\omega)}},\frac{1}{2}\}\int_\omega\left(\lVert
	\mathcal{G}_{\rm{Koiter}}^{\rm{lin}}(v) \rVert^2+\lVert \mathcal{R}_{\rm{Koiter}}^{\rm{lin}}({v}) \rVert^2\right){\rm da}. \end{align}
	Thus, there exists the positive constant $c>0$ such that
		\begin{align}\label{u9nnn}	\lVert
	\mathcal{G}_{\rm{Koiter}}^{\rm{lin}}(v) \rVert^2_{{\rm L}^2(\omega)}&+\lVert \mathcal{R}_{\rm{Koiter}}^{\rm{lin}}({v})-2\,\mathcal{G}_{\rm{Koiter}}^{\rm{lin}}({v}) \,{\rm L}_{y_0} \rVert^2_{{\rm L}^2(\omega)}\geq c\,\left(\lVert
	\mathcal{G}_{\rm{Koiter}}^{\rm{lin}}(v) \rVert^2_{{\rm L}^2(\omega)}+\lVert \mathcal{R}_{\rm{Koiter}}^{\rm{lin}}({v}) \rVert^2_{{\rm L}^2(\omega)}\right). \end{align}
	The third difference is that the Korn inequality  for  shells whose middle surface has little regularity given by Theorem \ref{Kornlr2} is now applicable, instead of the Korn inequality on general surfaces.
	\end{proof}

\subsection{The constrained linear $O(h^3)$-Cosserat shell model. Conditional existence }
By ignoring the $O(h^5)$ terms from the functional defining the variational problem of the  constrained linear $O(h^3)$-Cosserat shell model, we obtain the linearised  constrained Cosserat $O(h^3)$-shell model which   is  to find the midsurface displacement vector field
$v:\omega\subset\mathbb{R}^2\to\mathbb{R}^3$  minimizing on $\omega$:
\begin{align}\label{minvarlch3con}
I= \int_{\omega}   \,\, \Big[ & \Big(h+{\rm K}\,\dfrac{h^3}{12}\Big)\,
W_{{\rm shell}}^{\infty}\big([\nabla\Theta \,]^{-T}
[\mathcal{G}_{\rm{Koiter}}^{\rm{lin}}]^\flat [\nabla\Theta \,]^{-1} \big)\vspace{2.5mm}\notag\\&+   \dfrac{h^3}{12}\,
W_{{\rm shell}}^{\infty}  \big( [\nabla\Theta \,]^{-T} [\mathcal{R}_{\rm{Koiter}}^{\rm{lin}}-2\,\mathcal{G}_{\rm{Koiter}}^{\rm{lin}} \,{\rm L}_{y_0} ]^\flat
[\nabla\Theta \,]^{-1}\big) 
\vspace{2.5mm}\\&+\dfrac{h^3}{3} \mathrm{ H}\,\mathcal{W}_{{\rm shell}}^{\infty}  \big(  [\nabla\Theta \,]^{-T}
[\mathcal{G}_{\rm{Koiter}}^{\rm{lin}}]^\flat [\nabla\Theta \,]^{-1} , [\nabla\Theta \,]^{-T} [\mathcal{R}_{\rm{Koiter}}^{\rm{lin}}-2\,\mathcal{G}_{\rm{Koiter}}^{\rm{lin}} \,{\rm L}_{y_0} ]^\flat
[\nabla\Theta \,]^{-1} \big)\notag\\&-
\dfrac{h^3}{6}\, \mathcal{W}_{{\rm shell}}^{\infty}  \big(  [\nabla\Theta \,]^{-T}
[\mathcal{G}_{\rm{Koiter}}^{\rm{lin}}]^\flat [\nabla\Theta \,]^{-1} , [\nabla\Theta \,]^{-T}
[(\mathcal{R}_{\rm{Koiter}}^{\rm{lin}}-2\,\mathcal{G}_{\rm{Koiter}}^{\rm{lin}} \,{\rm L}_{y_0})\,{\rm L}_{y_0}]^\flat [\nabla\Theta \,]^{-1}\big)\vspace{2.5mm}\notag\\&+ \Big(h-{\rm K}\,\dfrac{h^3}{12}\Big)\,
W_{\mathrm{curv}}\big( (\nabla\vartheta_\infty \, |\, 0) \; [\nabla\Theta \,]^{-1} \big)    \vspace{2.5mm}\notag\\&+  \dfrac{h^3}{12}\,
W_{\mathrm{curv}}\big(  (\nabla\vartheta_\infty \, |\, 0) \; {\rm L}_{y_0}^\flat[\nabla\Theta \,]^{-1} \big)  
\Big] \,{\rm det}(\nabla y_0|n_0)       \,\mathrm{\rm da}- \overline{\Pi}_\infty(v),\notag
\end{align}
such that
\begin{align}\label{contsyml2h3}
\mathcal{G}_{\rm{Koiter}}^{\rm{lin}} \,{\rm L}_{y_0} &\in{\rm \textbf{Sym}}(2) \qquad \textrm{and}\qquad 
(\mathcal{R}_{\rm{Koiter}}^{\rm{lin}}-2\,\mathcal{G}_{\rm{Koiter}}^{\rm{lin}} \,{\rm L}_{y_0})\,{\rm L}_{y_0}\in{\rm \textbf{Sym}}(2).
\end{align}

The coercivity result given by the following lemma, see \cite{GhibaNeffPartIII},  and similar arguments as in proving the existence results in the theory including terms up to order $O(h^5)$ lead to  {corresponding} existence  {results}.

In the  constrained nonlinear Cosserat shell model up to $O(h^3)$
the shell energy density $W^{(h^3)}(\mathcal{E}_{\infty}, \mathcal{K}_{\infty})$ is given by 
\begin{align}\label{h3energy} W^{(h^3)}(\mathcal{E}_{\infty}, \mathcal{K}_{\infty})=&\,  \Big(h+{\rm K}\,\dfrac{h^3}{12}\Big)\,
W_{\mathrm{shell}}^\infty\big(    \mathcal{E}_{\infty}\big)+  \dfrac{h^3}{12}\,
W_{\mathrm{shell}}^\infty  \big(   \mathcal{E}_{\infty} \, {\rm B}_{y_0} +   {\rm C}_{y_0} \mathcal{K}_{\infty} \big) \notag \\&
-\dfrac{h^3}{3} \mathrm{ H}\,\mathcal{W}_{\mathrm{shell}}^\infty  \big(  \mathcal{E}_{\infty} ,
\mathcal{E}_{\infty}{\rm B}_{y_0}+{\rm C}_{y_0}\, \mathcal{K}_{\infty} \big)+
\dfrac{h^3}{6}\, \mathcal{W}_{\mathrm{shell}}^\infty  \big(  \mathcal{E}_{\infty} ,
( \mathcal{E}_{\infty}{\rm B}_{y_0}+{\rm C}_{y_0}\, \mathcal{K}_{\infty}){\rm B}_{y_0} \big)\notag\vspace{2.5mm}\\
&+  \Big(h-{\rm K}\,\dfrac{h^3}{12}\Big)\,
W_{\mathrm{curv}}\big(  \mathcal{K}_{\infty} \big)    +  \dfrac{h^3}{12}
W_{\mathrm{curv}}\big(  \mathcal{K}_{\infty}   {\rm B}_{y_0} \,  \big).
\end{align}
\begin{lemma}\label{coerh3r}{\rm [Coercivity  in the theory including terms up to order $O(h^3)$]} Assume that the constitutive coefficients are  such that $\mu>0$, $2\,\lambda+\mu> 0$, $b_1>0$, $b_2>0$, $b_3>0$ and $L_{\rm c}>0$ and let $c_2^+$  denote the smallest eigenvalue  of
	$
	W_{\mathrm{curv}}(  S ),
	$
	and $c_1^+$ and $ C_1^+>0$ denote the smallest and the largest eigenvalues of the quadratic form $W_{\mathrm{shell}}^\infty(  S)$.
	If the thickness $h$ satisfies  one of the following conditions:
	\begin{align}\label{fcondh3b}
	&{\rm i)} \ h\max\{\sup_{x\in\omega}|\kappa_1|, \sup_{x\in\omega}|\kappa_2|\}<\alpha \quad \
	\text{\bf and} \quad  	h^2<\frac{(5-2\sqrt{6})(\alpha^2-12)^2}{4\, \alpha^2}\frac{ {c_2^+}}{C_1^+} \quad  \text{with}  \quad  \quad 0<\alpha<2;\notag\\\\
	&{\rm ii)} \ h\max\{\sup_{x\in\omega}|\kappa_1|, \sup_{x\in\omega}|\kappa_2|\}<\frac{1}{a}  \quad \text{\bf and} \quad  a>\max\Big\{1 + \frac{\sqrt{2}}{2},\frac{1+\sqrt{1+3\frac{C_1^+}{c_1^+}}}{2}\Big\},\notag
	\end{align} then
	the total energy density $W^{(h^3)}(\mathcal{E}_{\infty}, \mathcal{K}_{\infty})$
	is coercive, in the sense that  there exists   a constant $a_1^+>0$ such that 
	$	W^{(h^3)}(\mathcal{E}_{\infty}, \mathcal{K}_{\infty})\,\geq\, a_1^+\, \big(  \lVert \mathcal{E}_{\infty}\rVert ^2 + \lVert \mathcal{K}_{\infty}\rVert ^2\,\big) ,
	$ where
	$a_1^+$ depends on the constitutive coefficients. In fact,  the following inequality holds true
	\begin{align}
	W^{(h^3)}(\mathcal{E}_{\infty}, \mathcal{K}_{\infty})\,\geq\, a_1^+\, \big( \lVert \mathcal{E}_{\infty}\rVert ^2 +\lVert
	\mathcal{E}_{\infty}{\rm B}_{y_0}+{\rm C}_{y_0}\, \mathcal{K}_{\infty}  \rVert ^2+ \lVert \mathcal{K}_{\infty}\rVert ^2\,\big).
	\end{align}
\end{lemma}
Indeed, we have:

\begin{theorem}\label{th1condh3}{\rm [A conditional existence result for the theory including terms up to order $O(h^3)$  on a general surface]}
	Let there be given a domain $\omega\subset \mathbb{R}^2$ and an injective mapping $y_0\in {\rm C}^{3}(\overline{\omega}, \mathbb{R}^3)$ such that the two vectors $a_\alpha=\partial_{x_\alpha}y_0$, $\alpha=1,2$, are linear independent at all points of $\overline{\omega}$.	Assume that the  admissible set $\mathcal{A}^{\rm lin}_\infty$ is non-empty and that the linear operator
	$
	{\Pi}^{\rm lin}_\infty$ is bounded. 
	Assume that the constitutive coefficients are  such that $\mu>0$, $2\,\lambda+\mu> 0$, $b_1>0$, $b_2>0$, $b_3>0$ and $L_{\rm c}>0$ and let $c_2^+$  denote the smallest eigenvalue  of
	$
	W_{\mathrm{curv}}(  S ),
	$
	and $c_1^+$ and $ C_1^+>0$ denote the smallest and the largest eigenvalues of the quadratic form $W_{\mathrm{shell}}^\infty(  S)$.
 {	If the thickness $h$ satisfies  one of the  conditions \eqref{fcondh3b},}
 then the  problem \eqref{minvarlch3con} admits a unique solution
	$v\in  \mathcal{A}^{\rm lin}_\infty$.
\end{theorem}

\begin{theorem}\label{th1cond2h3}{\rm [A conditional existence result for the theory including terms up to order $O(h^3)$  for shells whose middle surface has little regularity]}\\
	Let there be given a domain $\omega\subset \mathbb{R}^2$ and an injective mapping $y_0\in {\rm H}^{2,\infty}(\omega, \mathbb{R}^3)$ such that the two vectors $a_\alpha=\partial_{x_\alpha}y_0$, $\alpha=1,2$, are linear independent at all points of $\overline{\omega}$. 	Assume that the admissible set $\widehat{\mathcal{A}}^{\rm lin}_\infty$ is non-empty and that the linear operator
	$
	{\Pi}^{\rm lin}_\infty$ is bounded.	
	Assume that the constitutive coefficients are  such that $\mu>0$, $2\,\lambda+\mu> 0$, $b_1>0$, $b_2>0$, $b_3>0$ and $L_{\rm c}>0$ and let $c_2^+$  denote the smallest eigenvalue  of
	$
	W_{\mathrm{curv}}(  S ),
	$
	and $c_1^+$ and $ C_1^+>0$ denote the smallest and the largest eigenvalues of the quadratic form $W_{\mathrm{shell}}^\infty(  S)$.
	 {	If the thickness $h$ satisfies  one of the  conditions \eqref{fcondh3b},}  then the  problem \eqref{minvarlch3con} admits a unique solution
	$v\in  \widehat{\mathcal{A}}^{\rm lin}_\infty$.
\end{theorem}
\begin{remark}
 {As we have already mentioned for the nonlinear Cosserat-shell models \cite{GhibaNeffPartII,GhibaNeffPartIII}, a large value for $\alpha$ will relax the first condition \eqref{fcondh3b}$_1$ while the other condition \eqref{fcondh3b}$_2$ on the thickness will remain more restrictive. While for the $O(h^5)$-model the conditions imposed on the thickness do  not depend on the constitutive paramentes, in $O(h^5)$-model the conditions \eqref{fcondh3b} are  expressed in terms of all  constitutive parameters, through $c_1^+$, $c_2^+$ and $C_1^+$.}
\end{remark}
\section{A modified constrained linear $O(h^5)$-Cosserat shell model. Unconditional existence }\label{sec:modified}\setcounter{equation}{0}

In order to present an unconditional existence result, independent of whether the admissible set is non empty, we need to modify the model slightly. We have already done so in the geometrically nonlinear case \cite{GhibaNeffPartIII} and here we just  adapt that procedure properly to the linearised case. 
\subsection{Variational problem for the modified constrained non-linear Cosserat shell model}
In the modified constrained Cosserat shell model	the (through the thickness reconstructed) strain tensor  is considered to be \begin{align}\label{extEm}
\widetilde{\mathcal{E}}_{s}  \; =\; &\,\quad \,\,
1\,\Big[  \mathcal{E}_{ \infty } - \frac{\lambda}{\lambda+2\mu}\,\tr( \mathcal{E}_{ \infty } )\; (0|0|n_0)\, (0|0|n_0)^T  \Big]
\notag\vspace{2.5mm}\\
& 
+x_3\Big[ {\rm \textbf{sym}}(\mathcal{E}_{ \infty } \, {\rm B}_{y_0} +  {\rm C}_{y_0} \mathcal{K}_{ \infty }) -
\frac{\lambda}{(\lambda+2\mu)}\, {\rm tr}  \, [{\rm \textbf{sym}}(\mathcal{E}_{ \infty } {\rm B}_{y_0} + {\rm C}_{y_0}\mathcal{K}_{ \infty } )]\; (0|0|n_0)\,  (0|0|n_0)^T  \Big]
\vspace{2.5mm}\\
& 
+x_3^2\,{\rm \textbf{sym}}\Big[\,(\mathcal{E}_{ \infty } \, {\rm B}_{y_0} +  {\rm C}_{y_0} \mathcal{K}_{ \infty }) {\rm B}_{y_0} \Big]\;+\; O(x_3^3).\notag
\end{align}
which is now symmetric by definition through the additional by  applied \textbf{symmetrization}.

This ansatz of 	the (through the thickness reconstructed) strain tensor  leads to a model in which only the symmetric parts of $\mathcal{E}_{ \infty } \, {\rm B}_{y_0} +  {\rm C}_{y_0} \mathcal{K}_{ \infty }$ and $(\mathcal{E}_{ \infty } \, {\rm B}_{y_0} +  {\rm C}_{y_0} \mathcal{K}_{ \infty }) {\rm B}_{y_0}$ are involved and there will be no need to assume  a priori  that they have to be symmetric. 

	The  resulting variational problem for the modified constrained Cosserat $O(h^5)$-shell model \cite{GhibaNeffPartIII} is  to find a deformation of the midsurface
$m:\omega\subset\mathbb{R}^2\to\mathbb{R}^3$  minimizing on $\omega$:
\begingroup
\begin{align}\label{modminvarmc}
\hspace*{-2.5cm}I= \int_{\omega}   \,\, \Big[ & \,\Big(h+{\rm K}\,\dfrac{h^3}{12}\Big)\,
W_{{\rm shell}}^{\infty}\big(    \sqrt{[\nabla\Theta ]^{-T}\,\widehat{\rm I}_m\,\id_2^{\flat }\,[\nabla\Theta ]^{-1}}-
\sqrt{[\nabla\Theta ]^{-T}\,\widehat{\rm I}_{y_0}\,\id_2^{\flat }\,[\nabla\Theta ]^{-1}} \big),\vspace{2.5mm}\notag\\ \hspace*{-2.5cm}   
& +   \Big(\dfrac{h^3}{12}\,-{\rm K}\,\dfrac{h^5}{80}\Big)\,
W_{{\rm shell}}^{\infty}  \Big({\rm \textbf{sym}} \Big[   \sqrt{[\nabla\Theta ]^{-T}\,\widehat{\rm I}_m\,[\nabla\Theta ]^{-1}}\,  [\nabla\Theta ]\Big({\rm L}_{y_0}^\flat - {\rm L}_m^\flat\Big)[\nabla\Theta ]^{-1}\Big]\Big) \notag \\\hspace*{-2.5cm}&
-\dfrac{h^3}{3} \mathrm{ H}\,\mathcal{W}_{{\rm shell}}^{\infty}  \Big(  \sqrt{[\nabla\Theta ]^{-T}\,\widehat{\rm I}_m\,\id_2^{\flat }\,[\nabla\Theta ]^{-1}}-
\sqrt{[\nabla\Theta ]^{-T}\,\widehat{\rm I}_{y_0}\,\id_2^{\flat }\,[\nabla\Theta ]^{-1}} ,\notag\\\hspace*{-2.5cm}
&\qquad \qquad \qquad\quad\   {\rm \textbf{sym}} \Big[   \sqrt{[\nabla\Theta ]^{-T}\,\widehat{\rm I}_m\,[\nabla\Theta ]^{-1}}\,  [\nabla\Theta ]\Big({\rm L}_{y_0}^\flat - {\rm L}_m^\flat\Big)[\nabla\Theta ]^{-1}\Big]\Big)\\\hspace*{-2.5cm}&+
\dfrac{h^3}{6}\, \mathcal{W}_{{\rm shell}}^{\infty}  \Big(  \sqrt{[\nabla\Theta ]^{-T}\,\widehat{\rm I}_m\,\id_2^{\flat }\,[\nabla\Theta ]^{-1}}-
\sqrt{[\nabla\Theta ]^{-T}\,\widehat{\rm I}_{y_0}\,\id_2^{\flat }\,[\nabla\Theta ]^{-1}} ,
\notag\\\hspace*{-2.5cm}
&\qquad \qquad \qquad\ \ {\rm \textbf{sym}}\Big[\sqrt{[\nabla\Theta ]^{-T}\,\widehat{\rm I}_m\,[\nabla\Theta ]^{-1}}\,  [\nabla\Theta ]\Big({\rm L}_{y_0}^\flat - {\rm L}_m^\flat\Big){\rm L}_{y_0}^\flat[\nabla\Theta ]^{-1}\Big]\Big)\vspace{2.5mm}\notag\\
&+ \,\dfrac{h^5}{80}\,\,
W_{\mathrm{mp}}^{\infty} \Big({\rm \textbf{sym}} \Big[\sqrt{[\nabla\Theta ]^{-T}\,\widehat{\rm I}_m\,[\nabla\Theta ]^{-1}}\,  [\nabla\Theta ]\Big({\rm L}_{y_0}^\flat - {\rm L}_m^\flat\Big){\rm L}_{y_0}^\flat[\nabla\Theta ]^{-1}\,\Big]\Big)\notag\vspace{2.5mm}\notag\\
&+ \Big(h-{\rm K}\,\dfrac{h^3}{12}\Big)\,
W_{\mathrm{curv}}\big(  \mathcal{K}_{ \infty } \big)    +  \Big(\dfrac{h^3}{12}\,-{\rm K}\,\dfrac{h^5}{80}\Big)\,
W_{\mathrm{curv}}\big(  \mathcal{K}_{ \infty }   {\rm B}_{y_0} \,  \big)  + \,\dfrac{h^5}{80}\,\,
W_{\mathrm{curv}}\big(  \mathcal{K}_{ \infty }   {\rm B}_{y_0}^2  \big)
\Big] \,{\rm det}\nabla \Theta        \,\mathrm{\rm da}\notag\\&\quad - \overline{\Pi}(m,{Q}_{ \infty }).\notag
\end{align}
\endgroup
The admissible set  $\mathcal{A}^{\rm mod}$  is  defined by\footnote{The definition of the admissible set $\mathcal{A}^{\rm mod}$  incorporates a weak reformulation of the imposed symmetry constraint $\mathcal{E}_{\infty} \in {\rm Sym}(3)$. We notice that the constraints $U\coloneqq {Q}_{ \infty }^T (\nabla  m|{Q}_{ \infty }Q_0.e_3)[\nabla\Theta ]^{-1} \in {\rm L^2}(\omega, {\rm Sym}^+(3))$ together with the compatibility conditions between ${Q}_{ \infty }\Big|_{\gamma_d}$ and the values of $m$ on ${\gamma_d}$ will imply that $Q_{\infty}$ and $m$ are not independent variables and
	$
	{Q}_{ \infty }={\rm polar}\big[(\nabla  m|n) [\nabla\Theta ]^{-1}\big]\in\textrm{SO}(3),
	$ 
	where $n=\frac{\partial_{x_1} m\times \partial_{x_2} m}{\lVert\partial_{x_1} m\times \partial_{x_2} m\rVert}$ is the unit normal vector to the deformed midsurface. Assuming that 	the boundary data satisfy the conditions 
	${m}^*\in{\rm H}^1(\omega ,\mathbb{R}^3)$ and ${\rm polar}(\nabla {m}^*\,|\,n^*)\in{\rm H}^1(\omega, {\rm SO}(3))$, it follows that the admissible set is not empty.}
\begin{align}
\mathcal{A}^{\rm mod}=\Bigg\{(m,&{Q}_{ \infty })\in{\rm H}^1(\omega, \mathbb{R}^3)\times{\rm H}^1(\omega, {\rm SO}(3))\ \bigg| \  m\big|_{ \gamma_d}=m^*, \qquad{Q}_{ \infty }Q_0.e_3\big|_{ \gamma_d}=\,\dd\frac{\partial_{x_1}m^*\times \partial_{x_2}m^*}{\lVert \partial_{x_1}m^*\times \partial_{x_2}m^*\rVert }\notag\\ 
&\ U\coloneqq {Q}_{ \infty }^T (\nabla  m|{Q}_{ \infty }Q_0.e_3)[\nabla\Theta ]^{-1} \in {\rm L^2}(\omega, {\rm Sym}^+(3)),
\Bigg\}\notag
\end{align}
where 
\begingroup
\allowdisplaybreaks
\begin{align} 
\mathcal{K}_{ \infty } & = \, \Big(\mathrm{axl}({Q}_{ \infty }^T\,\partial_{x_1} {Q}_{ \infty })\,|\, \mathrm{axl}({Q}_{ \infty }^T\,\partial_{x_2} {Q}_{ \infty })\,|0\Big)[\nabla\Theta ]^{-1}, \vspace{2.5mm}\notag\\
{Q}_{ \infty }&={\rm polar}\big((\nabla  m|n) [\nabla\Theta ]^{-1}\big)=(\nabla m|n)[\nabla\Theta ]^{-1}\,\sqrt{[\nabla\Theta ]\,\widehat {\rm I}_{m}^{-1}\,[\nabla\Theta ]^{T}},\notag\\
W_{{\rm shell}}^{\infty}(  S)  &=   \mu\,\lVert\,   S\rVert^2  +\,\dfrac{\lambda\,\mu}{\lambda+2\mu}\,\big[ \mathrm{tr}   \, (S)\big]^2,\qquad 
\mathcal{W}_{{\rm shell}}^{\infty}(  S,  T) =   \mu\,\bigl\langle  S,   T\bigr\rangle+\,\dfrac{\lambda\,\mu}{\lambda+2\mu}\,\mathrm{tr}  (S)\,\mathrm{tr}  (T), \vspace{2.5mm}\\
W_{\mathrm{mp}}^{\infty}(  S)&= \mu\,\lVert  S\rVert^2+\,\dfrac{\lambda}{2}\,\big[ \mathrm{tr}\,   (S)\big]^2 \qquad \quad\forall \ S,T\in{\rm Sym}(3), \notag\vspace{2.5mm}\\
W_{\mathrm{curv}}(  X )&=\mu\,{\rm L}_c^2\left( b_1\,\lVert \dev\,\sym \, X\rVert^2+b_2\,\lVert\skw \,X\rVert^2+b_3\,
[\tr(X)]^2\right) \quad\qquad \forall\  X\in\mathbb{R}^{3\times 3}.\notag
\end{align}
\endgroup

\subsection{The modified constrained linear $O(h^5)$-Cosserat shell model}
Taking into account the modified constitutive nonlinear $O(h^5)$-Cosserat shell model,
the  variational problem for the constrained Cosserat $O(h^5)$-shell linear model  is now  to find a deformation of the midsurface
$v:\omega\subset\mathbb{R}^2\to\mathbb{R}^3$  minimizing on $\omega$:
\begin{align}\label{minvarlcmod}
I= \int_{\omega}   \,\, \Big[ & \Big(h+{\rm K}\,\dfrac{h^3}{12}\Big)\,
W_{{\rm shell}}^{\infty}\big([\nabla\Theta \,]^{-T}
[\mathcal{G}_{\rm{Koiter}}^{\rm{lin}}]^\flat [\nabla\Theta \,]^{-1} \big)\vspace{2.5mm}\notag\\&+   \Big(\dfrac{h^3}{12}\,-{\rm K}\,\dfrac{h^5}{80}\Big)\,
W_{{\rm shell}}^{\infty}  \big( [\nabla\Theta \,]^{-T} {\rm \textbf{sym}}[\mathcal{R}_{\rm{Koiter}}^{\rm{lin}}-2\,\mathcal{G}_{\rm{Koiter}}^{\rm{lin}} \,{\rm L}_{y_0} ]^\flat
[\nabla\Theta \,]^{-1}\big) 
\vspace{2.5mm}\notag\\&+\dfrac{h^3}{3} \mathrm{ H}\,\mathcal{W}_{{\rm shell}}^{\infty}  \big(  [\nabla\Theta \,]^{-T}
[\mathcal{G}_{\rm{Koiter}}^{\rm{lin}}]^\flat [\nabla\Theta \,]^{-1} , [\nabla\Theta \,]^{-T} {\rm \textbf{sym}}[\mathcal{R}_{\rm{Koiter}}^{\rm{lin}}-2\,\mathcal{G}_{\rm{Koiter}}^{\rm{lin}} \,{\rm L}_{y_0} ]^\flat
[\nabla\Theta \,]^{-1} \big)\notag\\&-
\dfrac{h^3}{6}\, \mathcal{W}_{{\rm shell}}^{\infty}  \big(  [\nabla\Theta \,]^{-T}
[\mathcal{G}_{\rm{Koiter}}^{\rm{lin}}]^\flat [\nabla\Theta \,]^{-1} , [\nabla\Theta \,]^{-T}
{\rm \textbf{sym}}[(\mathcal{R}_{\rm{Koiter}}^{\rm{lin}}-2\,\mathcal{G}_{\rm{Koiter}}^{\rm{lin}} \,{\rm L}_{y_0})\,{\rm L}_{y_0}]^\flat [\nabla\Theta \,]^{-1}\big)\vspace{2.5mm}\notag\\&+ \,\dfrac{h^5}{80}\,\,
W_{\mathrm{mp}}^{\infty} \big([\nabla\Theta \,]^{-T}
{\rm \textbf{sym}}[(\mathcal{R}_{\rm{Koiter}}^{\rm{lin}}-2\,\mathcal{G}_{\rm{Koiter}}^{\rm{lin}} \,{\rm L}_{y_0})\,{\rm L}_{y_0} ]^\flat [\nabla\Theta \,]^{-1}\big)\vspace{2.5mm}\\&+ \Big(h-{\rm K}\,\dfrac{h^3}{12}\Big)\,
W_{\mathrm{curv}}\big( (\nabla\vartheta_\infty \, |\, 0) \; [\nabla\Theta \,]^{-1} \big)    \vspace{2.5mm}\notag\\&+  \Big(\dfrac{h^3}{12}\,-{\rm K}\,\dfrac{h^5}{80}\Big)\,
W_{\mathrm{curv}}\big(  (\nabla\vartheta_\infty \, |\, 0) \; {\rm L}_{y_0}^\flat[\nabla\Theta \,]^{-1} \big)  \vspace{2.5mm}\notag\\&+ \,\dfrac{h^5}{80}\,\,
W_{\mathrm{curv}}\big(   (\nabla\vartheta_\infty \, |\, 0) \; ({\rm L}_{y_0}^\flat)^2[\nabla\Theta \,]^{-1}  \big)\notag
\Big] \,{\rm det}(\nabla y_0|n_0)       \,\mathrm{\rm da}- \widetilde{\Pi}(u),\notag
\end{align}
where 
\begin{align}
W_{{\rm shell}}^{\infty}(  S)  &=   \mu\,\lVert\,   S\rVert^2  +\,\dfrac{\lambda\,\mu}{\lambda+2\mu}\,\big[ \mathrm{tr}   \, (S)\big]^2,\qquad \qquad \qquad 
\mathcal{W}_{{\rm shell}}^{\infty}(  S,  T) =   \mu\,\bigl\langle  S,   T\bigr\rangle+\,\dfrac{\lambda\,\mu}{\lambda+2\mu}\,\mathrm{tr}  (S)\,\mathrm{tr}  (T), \notag\vspace{2.5mm}\\
W_{\mathrm{mp}}^{\infty}(  S)&= \mu\,\lVert  S\rVert^2+\,\dfrac{\lambda}{2}\,\big[ \mathrm{tr}\,   (S)\big]^2 \quad \forall \ S,T\in{\rm Sym}(3), \notag\vspace{2.5mm}\\
W_{\mathrm{curv}}(  X )&=\mu\,{\rm L}_c^2\left( b_1\,\lVert \dev\,\sym \, X\rVert^2+b_2\,\lVert\skw \,X\rVert^2+b_3\,
[\tr(X)]^2\right) \quad \forall\  X\in\mathbb{R}^{3\times 3}.\notag
\end{align}

\subsection{Unconditional existence results}

The  sets  of admissible functions are accordingly  defined by
\begin{align}\label{21lu}
\mathcal{A}_{\rm lin}^{\rm mod}=\Bigg\{v=(v_1,v_2,v_3)&\in {\rm H}^1(\omega,\mathbb{R})\times{\rm H}^1(\omega,\mathbb{R})\times {\rm H}^2(\omega,\mathbb{R}) \,\big|\,  v_1=v_2=v_3=\bigl\langle \nabla v_3,\nu\bigr\rangle=0 \ \ \text{on}\ \ \partial \omega,\notag\\&\nabla [\skw( (\nabla v\,|\dd\sum_{\alpha=1,2}\bigl\langle n_0, \partial_{x_\alpha}v\bigr\rangle\, a^\alpha)[\nabla \Theta]^{-1})]\in {\rm L}^2(\omega)
\Bigg\},
\end{align}
and
\begin{align}\label{21l2u}
\widehat{\mathcal{A}}_{\rm lin}^{\rm mod}=\Bigg\{v\in {\rm H}_0^1(\omega,\mathbb{R}^3)&\,|\, \bigl\langle \partial_{x_\alpha x_\beta}v, n_0\bigr\rangle\in {\rm L}^2(\omega),\ \ \ \nabla [\skw( (\nabla v\,|\dd\sum_{\alpha=1,2}\bigl\langle n_0, \partial_{x_\alpha}v\bigr\rangle\, a^\alpha)[\nabla \Theta]^{-1})]\in {\rm L}^2(\omega)
\Bigg\},
\end{align}
depending on the expressions of $\mathcal{G}_{\rm{Koiter}}^{\rm{lin}}$ and $
\mathcal{R}_{\rm{Koiter}}^{\rm{lin}}$ which we consider, i.e.,  \eqref{formK} and \eqref{formR} \textbf{or} \eqref{equ11} and \eqref{equ12}, respectively. 

\textit{Hence, we have avoided the problem that the new set   of admissible functions may be non-empty} and, since all our inequalities from the proof of the conditional existence result involve only symmetric matrices, we have the following unconditional existence result:

\begin{theorem}\label{th1uncond}{\rm [Unconditional existence result for the theory including terms up to order $O(h^5)$  on a general surface]}
	Let there be given a domain $\omega\subset \mathbb{R}^2$ and an injective mapping $y_0\in {\rm C}^{3}(\overline{\omega}, \mathbb{R}^3)$ such that the two vectors $a_\alpha=\partial_{x_\alpha}y_0$, $\alpha=1,2$, are linear independent at all points of $\overline{\omega}$.	Assume that the linear operator
	$
	{\Pi}^{\rm lin}_\infty$ is bounded.	
	Then, for sufficiently small values of the thickness $h$ such that  {such that condition \eqref{rcondh5} is satisfied}
	and for constitutive coefficients  such that $\mu>0$, $2\,\lambda+\mu> 0$, $b_1>0$, $b_2>0$ and $b_3>0$, the  problem \eqref{minvarlcmod} admits a unique solution
	$v\in  \mathcal{A}_{\rm lin}^{\rm mod}$.
\end{theorem}

\begin{theorem}\label{th1uncond2}{\rm [Unconditional existence result for the theory including terms up to order $O(h^5)$  for shells whose middle surface has little regularity]}
	Let there be given a domain $\omega\subset \mathbb{R}^2$ and an injective mapping $y_0\in {\rm H}^{2,\infty}(\omega, \mathbb{R}^3)$ such that the two vectors $a_\alpha=\partial_{x_\alpha}y_0$, $\alpha=1,2$, are linear independent at all points of $\overline{\omega}$. 	Assume  that the linear operator
	$
	{\Pi}^{\rm lin}_\infty$ is bounded.	
	Then, for sufficiently small values of the thickness $h$  {such that condition \eqref{rcondh5} is satisfied}
	and for constitutive coefficients  such that $\mu>0$, $2\,\lambda+\mu> 0$, $b_1>0$, $b_2>0$ and $b_3>0$, the  problem \eqref{minvarlcmod} admits a unique solution
	$v\in  \widehat{\mathcal{A}}_{\rm lin}^{\rm mod}$.
\end{theorem}

By ignoring the $O(h^5)$-terms from the functional defining the variational problem, we obtain the linearised  constrained Cosserat $O(h^3)$-shell model which   is  to find the midsurface displacement vector field
$v:\omega\subset\mathbb{R}^2\to\mathbb{R}^3$  minimizing on $\omega$:
\begin{align}\label{minvarlch3um}
I= \int_{\omega}   \,\, \Big[ & \Big(h+{\rm K}\,\dfrac{h^3}{12}\Big)\,
W_{{\rm shell}}^{\infty}\big([\nabla\Theta \,]^{-T}
[\mathcal{G}_{\rm{Koiter}}^{\rm{lin}}]^\flat [\nabla\Theta \,]^{-1} \big)\vspace{2.5mm}\notag\\&+   \dfrac{h^3}{12}\,
W_{{\rm shell}}^{\infty}  \big( [\nabla\Theta \,]^{-T} {\rm \textbf{sym}}\Big[\mathcal{R}_{\rm{Koiter}}^{\rm{lin}}-2\,\mathcal{G}_{\rm{Koiter}}^{\rm{lin}} \,{\rm L}_{y_0} \Big]^\flat
[\nabla\Theta \,]^{-1}\big) 
\vspace{2.5mm}\notag\\&+\dfrac{h^3}{3} \mathrm{ H}\,\mathcal{W}_{{\rm shell}}^{\infty}  \big(  [\nabla\Theta \,]^{-T}
[\mathcal{G}_{\rm{Koiter}}^{\rm{lin}}]^\flat [\nabla\Theta \,]^{-1} , [\nabla\Theta \,]^{-T} {\rm \textbf{sym}}\Big[\mathcal{R}_{\rm{Koiter}}^{\rm{lin}}-2\,\mathcal{G}_{\rm{Koiter}}^{\rm{lin}} \,{\rm L}_{y_0} \Big]^\flat
[\nabla\Theta \,]^{-1} \big)\notag\\&-
\dfrac{h^3}{6}\, \mathcal{W}_{{\rm shell}}^{\infty}  \big(  [\nabla\Theta \,]^{-T}
[\mathcal{G}_{\rm{Koiter}}^{\rm{lin}}]^\flat [\nabla\Theta \,]^{-1} , [\nabla\Theta \,]^{-T}
{\rm \textbf{sym}}\Big[(\mathcal{R}_{\rm{Koiter}}^{\rm{lin}}-2\,\mathcal{G}_{\rm{Koiter}}^{\rm{lin}} \,{\rm L}_{y_0})\,{\rm L}_{y_0}\Big]^\flat [\nabla\Theta \,]^{-1}\big)\\&+ \Big(h-{\rm K}\,\dfrac{h^3}{12}\Big)\,
W_{\mathrm{curv}}\big( (\nabla\vartheta_\infty \, |\, 0) \; [\nabla\Theta \,]^{-1} \big)    \vspace{2.5mm}\notag\\&+  \Big(\dfrac{h^3}{12}\,-{\rm K}\,\dfrac{h^5}{80}\Big)\,
W_{\mathrm{curv}}\big(  (\nabla\vartheta_\infty \, |\, 0) \; {\rm L}_{y_0}^\flat[\nabla\Theta \,]^{-1} \big)  
\Big] \,{\rm det}(\nabla y_0|n_0)       \,\mathrm{\rm da}- \widetilde{\Pi}(u).\notag
\end{align}

Accordingly, the following results are true:

\begin{theorem}\label{th1condh3u}{\rm [Unconditional existence result for the theory including terms up to order $O(h^3)$  on a general surface]}
	Let there be given a domain $\omega\subset \mathbb{R}^2$ and an injective mapping $y_0\in {\rm C}^{3}(\overline{\omega}, \mathbb{R}^3)$ such that the two vectors $a_\alpha=\partial_{x_\alpha}y_0$, $\alpha=1,2$, are linear independent at all points of $\overline{\omega}$.	Assume that the linear operator
	$
	{\Pi}^{\rm lin}_\infty$ is bounded.	
	Assume that the constitutive coefficients are  such that $\mu>0$, $2\,\lambda+\mu> 0$, $b_1>0$, $b_2>0$, $b_3>0$ and $L_{\rm c}>0$ and let $c_2^+$  denote the smallest eigenvalue  of
	$
	W_{\mathrm{curv}}(  S ),
	$
	and $c_1^+$ and $ C_1^+>0$ denote the smallest and the largest eigenvalues of the quadratic form $W_{\mathrm{shell}}^\infty(  S)$.
	 {	If the thickness $h$ satisfies  one of the  conditions \eqref{fcondh3b},} then the  problem \eqref{minvarlch3um} admits a unique solution
	$v\in  \mathcal{A}^{\rm lin}_{\rm mod}$.
\end{theorem}

\begin{theorem}\label{th1cond2h3u}{\rm [Unconditional existence result for the theory including terms up to order $O(h^3)$  for shells whose middle surface has little regularity]}
	Let there be given a domain $\omega\subset \mathbb{R}^2$ and an injective mapping $y_0\in {\rm H}^{2,\infty}(\omega, \mathbb{R}^3)$ such that the two vectors $a_\alpha=\partial_{x_\alpha}y_0$, $\alpha=1,2$, are linear independent at all points of $\overline{\omega}$. 	Assume that the linear operator
	$
	{\Pi}^{\rm lin}_\infty$ is bounded.	
	Assume that the constitutive coefficients are  such that $\mu>0$, $2\,\lambda+\mu> 0$, $b_1>0$, $b_2>0$, $b_3>0$ and $L_{\rm c}>0$ and let $c_2^+$  denote the smallest eigenvalue  of
	$
	W_{\mathrm{curv}}(  S ),
	$
	and $c_1^+$ and $ C_1^+>0$ denote the smallest and the largest eigenvalues of the quadratic form $W_{\mathrm{shell}}^\infty(  S)$.
	 {	If the thickness $h$ satisfies  one of the  conditions \eqref{fcondh3b},}  then the  problem \eqref{minvarlch3um} admits a unique solution
	$v\in  \widehat{\mathcal{A}}^{\rm lin}_{\rm mod}$.
\end{theorem}

\section{Conclusion}
We have considered the geometrically nonlinear constrained Cosserat shell model and we deduced it's linearization. The constrained model is interesting in it's own right since it provides a natural bridge to more classical shell models, in which no independent triad of directors is considered. We were able to show existence and uniqueness, using precisely those tools that are usually employed: the classical Korn's inequality on a surface. However, one important caveat remained: the results provide only a conditional existence, since the requirements on the admissible set are to strong.
Therefore, we were led to modify the variational problem. More precisely, we only consider the symmetric parts of some strain measures in the elastic energy and this allowed us to prove unconditional existence results. Sine shell theory is an approximation anyway, we believe that the modified approach is vastly superior and merits further investigation if one is interested in remaining close to more  classical shell models.

\bigskip

\begin{footnotesize}
	\noindent{\bf Acknowledgements:}   This research has been funded by the Deutsche Forschungsgemeinschaft (DFG, German Research Foundation) -- Project no. 415894848: NE 902/8-1 (P. Neff).

	\bibliographystyle{plain} %plain

\end{footnotesize}

\end{document}